\documentclass[a4paper,11pt,DIV=11,
abstract=on,
parskip=half,
]{scrartcl}
 
\usepackage[margin=2.5cm]{geometry}
\usepackage[utf8]{inputenc}
\usepackage{amsmath,amssymb,amsfonts,amsthm,mathtools}
\usepackage{graphicx}
\usepackage{caption,subcaption}
\usepackage{enumerate}
\usepackage{algorithm, algpseudocode,setspace}
\usepackage{color}
\usepackage{booktabs}
\usepackage{pifont}
\usepackage{url}
\usepackage{colonequals}
\usepackage{dsfont} 

\mathtoolsset{showonlyrefs}

\newtheorem{theorem}{Theorem}[section]
\newtheorem{lemma}[theorem]{Lemma}
\newtheorem{remark}[theorem]{Remark}
\newtheorem{definition}[theorem]{Definition}

\newtheorem{corollary}[theorem]{Corollary}
\newtheorem{proposition}[theorem]{Proposition}

\makeatletter
\def\thm@space@setup{%
	\thm@preskip=\parskip \thm@postskip=0pt
}
\makeatother

\normalfont

\newcommand{\RR}{\ensuremath{\mathbb{R}}}
\newcommand{\E}{\ensuremath{\mathbb{E}}}
\newcommand{\W}{\ensuremath{\mathcal{W}}}

\newcommand{\supp}{\textnormal{supp}}

\newcommand{\dx}{\,\mathrm{d}}
\newcommand{\tT}{\mathrm{T}}

\newcommand{\cmark}{\ding{51}}
\newcommand{\xmark}{\ding{55}}

\DeclareMathOperator*{\argmin}{argmin}
\DeclareMathOperator*{\argmax}{argmax}

\DeclareMathOperator*{\diag}{diag}

\newcommand{\Wtwo}{\mathcal{W}_2}
\newcommand{\sumitoN}{\sum_{i=1}^N}
\newcommand{\sumitoNlambdai}{\sumitoN \lambda_i}
\newcommand{\mean}{M_{\lambda, \pi}^p}
\newcommand{\mmean}{G_{\lambda, \pi}^p}

\newcommand\curwidth{0.38\textwidth}

\numberwithin{equation}{section} 
\AtBeginDocument{
	\label{CorrectFirstPageLabel}
	
}

\newif\iflong

\title{Simple Approximative Algorithms for Free-Support Wasserstein Barycenters}
\author{Johannes von Lindheim\thanks{Institute of Mathematics,
		Technische Universit\"at Berlin,
		Strasse des 17. Juni 136, 10587 Berlin, Germany,
		vonlindheim@tu-berlin.de}}
\date{\today}

\begin{document}
	
\maketitle

\begin{abstract}
	Computing Wasserstein barycenters of discrete measures has recently attracted 
	considerable attention due to its wide variety of applications in data science.
	In general, this problem is NP-hard, calling for practical approximative algorithms.
	In this paper, we analyze a well-known simple framework for approximating Wasserstein-$p$ barycenters, where we mainly consider the most common case $p=2$ and $p=1$, which is not as well discussed.
	The framework produces sparse support solutions and shows good numerical results in the free-support setting.
	Depending on the desired level of accuracy, this requires only $N-1$ or $N(N-1)/2$ standard two-marginal optimal transport (OT) computations between the $N$ input measures, respectively,
	which is fast, memory-efficient and easy to implement using any OT solver as a black box.
	What is more, these methods yield a relative error of at most $N$ and $2$, respectively, for both $p=1, 2$.
	We show that these bounds are practically sharp.
	In light of the hardness of the problem, it is not surprising that such guarantees cannot be close to optimality in general.
	Nevertheless, these error bounds usually turn out to be drastically lower for a given particular problem in practice and can be evaluated with almost no computational overhead, in particular without knowledge of the optimal solution.
	In our numerical experiments, this guaranteed errors of at most a few percent.
\end{abstract}



\section{Introduction}

Wasserstein barycenters are an increasingly popular application of optimal transport in data science \cite{AC11barycenters,AC11barycenters,PRV20oncomputation}.
They have nice mathematical properties, since they are the Fr\'echet means with respect to the Wasserstein distance \cite{frechetpersistence14,frechettemplates05T,frechet19procrustes}.
Their applications range from mixing textures \cite{HLPR21gotex,texturemix11}, 
stippling patterns and bidirectional reflectance distribution functions \cite{BPPH11networksimplex},
or color distributions and shapes \cite{convolutional15SPCetal} over averaging of sensor data \cite{sensors20H} to Bayesian statistics \cite{bayes18}, just to name a few.
For further reading, we refer to the surveys \cite{stataspects19PZ,PC19book}.

Unfortunately, Wasserstein barycenters
are in general hard to compute \cite{AB21nphard}.
Many algorithms restrict the support of the solution to a fixed set and minimize only over the weights.
Such methods include projected subgradient \cite{CD14fast}, 
iterative Bregman projections \cite{BCC15IBP}, (proximal) algorithms 
based on the latter \cite{complexitybarycenters19KTDDGU}, 
interior point methods \cite{GWXY19MAAIPM}, 
Gauss-Seidel based alternating direction of multipliers \cite{LJDK21linear}, 
multi-marginal Sinkhorn algorithms and its accelerated variants \cite{lin2019complexity}, 
debiased Sinkhorn barycenter algorithms \cite{JCG20debiased}, 
methods using the Wasserstein distance on a tree \cite{TSKRY21treesliced}, 
accelerated Bregman projections \cite{LHXCJ20fastIBP} and  methods based on mirror proximal maps or  
on a dual extrapolation scheme \cite{DT21complexitybounds}, among others.
While iterative Bregman projections are a standard benchmark that are hard to beat in terms of simplicity and speed, fixed-support methods applied on a grid suffer from the curse of dimensionality.

On the other hand, barycenters without such restriction are called free-support barycenters.
This approach can overcome the curse of dimensionality, since the optimal solution is sparse.
Free-support barycenters can be computed directly from the solution of the closely related multi-marginal optimal transport (MOT) problem.
The latter was originally introduced in \cite{GS98} in the continuous setting for squared Euclidean costs and further generalized in various ways, e.g., to entropy regularized \cite{flows19BCN,tree21HRCK} and unbalanced variants with non-exact marginal constraints \cite{UMOT21BLNS}.
The solution to MOT can be obtained by solving a linear program (LP) that unfortunately scales exponentially in $N$, however \cite{ABM16brute}.
Although there are exact polynomial-time methods for measures on $\RR^d$ for fixed $d$ \cite{AB21fixedd}, 
see also LP-based methods in \cite{ABM16brute,improvedLP20BP,columngen22BP},
these are not necessarily fast in practice and rather involved to implement. 
A remedy is to resort to approximative approaches, which include so far a Newton-approach that iteratively alternates between optimizing over the weights and supports \cite{CD14fast}, another LP-based method \cite{B20LPapprox}, an inexact proximal alternating minimization method \cite{iPAM21QP}, an iterative stochastic algorithm \cite{stoch18CCS} and the iterative swapping algorithm \cite{PRV20oncomputation}.
A free-support barycenter method based on the Frank--Wolfe algorithm is given in \cite{LSPC2019frankwolfe}.
Another method  in \cite{cont20LGYS} computes continuous barycenters using another way of parameterizing them.
For approaches for MOT similar to this paper, see \cite{mot2022vL}.
Further speedups can be obtained by subsampling the given measures \cite{HMZ20randomized} or dimensionality reduction of the support point clouds \cite{ISS21dimreduction}.

Despite the plethora of literature, many algorithms with low theoretical computational complexity or high accuracy solutions are rather involved to implement.
This impedes its actual usage and further research in practice.
To the best of our knowledge, there does not exist an algorithm that fulfills the following list of desiderata in the free-support setting:
\begin{itemize}
	\item simple to implement,
	\item sharp theoretical error bounds,
	\item sparse solutions, and
	\item good numerical results in practice.
\end{itemize}

The purpose of this paper is to show that all of these points can be achieved using one iteration of a simple well-known fixed-point algorithm, which only requires some off-the-shelve two-marginal OT solver as ingredients to its otherwise easy implementation.
Here we consider the cases $p=2$ and $p=1$, where the latter has received less attention in the literature so far.
One such fixed-point iteration consists in computing optimal transport plans from a given measure to the input measures and pushing each atom to the $p$-barycenter of its target locations.
For the cost of $N-1$ OT plans, this yields a relative error bound of $N$, or a $2$-approximation, respectively, when averaging over these results, which requires to solve $N(N-1)/2$ OT problems.
The key to these theoretical bounds is based on the observation that the they are already fulfilled for the input measures or their mixture, respectively, which we choose as initialization.
On the other hand, we show that the aforementioned fixed-point iteration guarantees to at least retain the current approximation quality, but improves it considerably in practice in the first step.

Note that other algorithms with an upper error bound of $2$ have been proposed in \cite{B20LPapprox} for $p=2$.
The basic algorithm produces a barycenter with support $\cup_{i=1}^N\supp(\mu^i)$ by solving an LP over its weights.
However, while this support choice leads to bad approximations in practice (consider, e.g., two distinct Dirac measures as input), for a merely theoretical $2$-approximation, no computation is necessary as mentioned above.
On the other hand, the implementation and proofs of the other algorithms in that paper with better results in practice are rather involved.

In view of the hardness of the Wasserstein barycenter problem \cite{AB21nphard}, it is clear that the derived relative error bounds cannot be close to $1$ for every set of inputs, unless P = NP.
However, the improvement made by one iteration is straightforward to evaluate in the proposed algorithms, such that it can output relative error bounds specific to the given problem without knowing the optimal solution.
We observe these resulting improved bounds to be close to $1$ in the numerical experiments.

This paper is organized as follows:
We introduce the Wasserstein barycenter problem and our notation in Section~\ref{sec:bary}.
In Section~\ref{sec:algorithms}, we state the algorithms considered in this paper.
In Section~\ref{sec:analysis}, we analyze their worst-case relative error.
In Section~\ref{sec:numerics}, we provide a comparison with other algorithms on a synthetic data set, a numerical exploration of Wasserstein-$1$ barycenters, and two applications of the discussed framework.
Concluding remarks are given in Section~\ref{sec:discussion}.

\section{Wasserstein Barycenter Problem} \label{sec:bary}
%
In the following, we denote by $\Vert \cdot \Vert$ the Euclidean norm on $\RR^d$ and by
$\mathcal P(\RR^d)$ the space of probability measures on $\mathbb R^d$.
Let $1\leq p < \infty$.
For two discrete measures
\[
\mu^1 = \sum_{k=1}^{n_1} \mu^1_k \delta (x^1_k), \quad 
\mu^2 = \sum_{l=1}^{n_2} \mu^2_l \delta (x^2_l), 
\]
the \emph{Wasserstein-$p$ distance} 
is defined by
\[
\W_p^p(\mu^1, \mu^2) 
= 
\min_{\pi\in\Pi(\mu^1, \mu^2)} \langle c_p, \pi \rangle,
\]
where $\langle c_p, \pi\rangle = \smash{\int_{\RR^d\times\RR^d}}c_p \dx \pi $ with $c_p(x,y) \coloneqq \Vert x-y\Vert^p$ and $\Pi(\mu^1, \mu^2)$ denotes the set of probability measures on $\RR^d \times \RR^d$
with marginals $\mu^1$ and $\mu^2$.
The above optimization problem is convex, but can have multiple minimizers $\pi$.

In this paper, we are given $N$ discrete probability measures $\mu^i\in \mathcal P(\RR^d)$ supported at $\supp(\mu^i) = \{ x^i_1, \dots, x^i_{n_i}\}$, where the $x^i_l$ are pairwise different for every $i$, i.e., 
\begin{equation}\label{eq:mu_def}
\mu^i= \sum_{l=1}^{n_i} \mu^i_l \delta(x^i_l),\quad i=1, \dots, N.
\end{equation}
Let $\Delta_{N} \coloneqq \{\lambda \in (0,1)^N: \sumitoNlambdai =1\}$ denote the open probability simplex.
For given weights 
$\lambda = (\lambda_1,\ldots,\lambda_N) \in \Delta_{N}$, 
we are interested in the computation of Wasserstein barycenters, which are the solutions to the optimization problem
\begin{equation}\label{eq:bary}
\min_{\nu \in \mathcal P(\RR^d)} \Psi_p(\nu) , \qquad 
\Psi_p(\nu) \coloneqq \sumitoNlambdai \W_p^p(\nu, \mu^i).
\end{equation}

The following theorem restates important results from \cite[Prop.~3]{CE10teams}, which connects barycenter problems with what is nowadays known as multi-marginal optimal transport, as well as \cite[Prop.~1, Thm.~2]{ABM16brute} and \cite[Thm.~1]{sparsity22FP} in our notation.

\begin{theorem}\label{thm:basic}
	The barycenter problem \eqref{eq:bary} has at least one optimal solution $\hat\nu$.
	Every optimal solution $\hat\nu$ fulfills
	\begin{equation}\label{eq:support_discrete}
	\supp(\hat\nu) \subseteq \Big\{ \sumitoNlambdai x^i : x^i\in \supp(\mu^i), \, i=1, \dots, N \Big\}.
	\end{equation}
	Moreover, there exists an optimal solution $\hat\nu$, such that
	\begin{equation}\label{eq:sparsity}
		\#\supp(\hat\nu) \leq \sumitoN n_i - N + 1.
	\end{equation}
\end{theorem}
\begin{proof}
	Note that \eqref{eq:support_discrete} is straightforward to obtain from the relation to multi-marginal optimal transport \cite[Prop.~3]{CE10teams}.
	In the special case $p=2$, the results from \cite{ABM16brute}, in particular \eqref{eq:sparsity}, can readily be generalized to arbitrary $\lambda\in \Delta_N$.
	For general $p\geq 1$ and barycenter problems with even more general cost functions, this follows from sparsity of multi-marginal optimal transport recently shown in \cite[Thm.~1]{sparsity22FP} in combination with \cite[Prop.~3]{CE10teams}.
\end{proof}


In particular, the theorem says that finding optimal Wasserstein barycenters is a discrete optimization problem over the weights of its finite support, which is contained in the convex hull of the supports of the $\mu^i$.
However, the number of possible support points scales exponentially in $N$.

\section{Algorithms for Barycenter Approximation}\label{sec:algorithms}

In this section, after motivating the main framework considered in this paper in Section~\ref{sec:motivation}, we discuss two more concrete configurations of it in Sections~\ref{sec:ref_alg} and \ref{sec:pairwise_alg}.

\subsection{Motivation}\label{sec:motivation}

In its core, the algorithms in this paper approximate barycenters by ``averaging optimal transport plans'' from a particular reference measure to the input measures in some sense.
This approach is well-known and comes in various flavors in the literature.
For example, it can be viewed through the lens of generalized geodesics in Wasserstein spaces \cite{gradflows08AGS} and
recent literature on linear optimal transport and relatives
\cite{GWLOT21BBS,HK-LOT22CCST,LOT13WSBOR,LOT20MDC,LOT21MC}.
On the other hand, in \cite{CD14fast}, one of the first papers on the numerical approximation of Wasserstein barycenters, the idea is presented as a Newton iteration.
The same iteration is analyzed in the continuous setting in \cite{ABC16fixedpoint}, and it can be used as a characterization of Wasserstein barycenters in terms of fixed points of this procedure, even for uncountably many input measures \cite{avgOTmaps18BK}.
See also \cite{weakOT21CTF} for this algorithm in the context of weak optimal transport.

Let us define the averaging of transport plans more precisely.
\begin{definition}
	Given a discrete measure $\nu = \sum_{k=1}^{n_\nu}\nu_k \delta(y_k) \in \mathcal P(\RR^d)$
	and transport plans
	\begin{equation}
	\pi^i \coloneqq \sum_{k=1}^{n_\nu}\sum_{l=1}^{n_i} \pi^i_{k, l} \delta( y_k, x_l^i) \in \Pi (\nu, \mu^i), \quad i=1, \dots, N,
	\end{equation}
	set $\pi=(\pi^1, \dots, \pi^N)$ and let for $k=1, \dots, n_\nu$, $p\geq 1$, the barycentric map $\mean\colon \supp(\nu)\to \RR^d$ be defined as
	\begin{equation}
	m_k = \mean(y_k) \coloneqq \argmin_{m\in \RR^d}\sumitoNlambdai \sum_{l=1}^{n_i} \frac{\pi^i_{k, l}}{\nu_k} \Vert m -  x_l^i \Vert^p.
	\end{equation}
	Furthermore, we define the mapping
	\begin{equation}\label{eq:mmean}
	\mmean(\nu) \coloneqq \sum_{k=1}^{n_\nu} \nu_k \delta(m_k).
	\end{equation}
\end{definition}
That is, each atom $y_k$ in the measure $\nu$ is pushed to the weighted barycenter $m_k$ of its target locations $x_l^i$, where the weights are given by the $\lambda_i$ and the weights of the source locations as given by the transport plans $\pi^i$, relative to the corresponding transported mass $\nu_k$.

Note that for $p=2$, the map $\mean$ is the classical mean, whereas for $p=1$, it is called geometric median.
It is uniquely defined, whenever the points are not collinear, see \cite{weiszfeld15BS}.
Otherwise, in case of ambiguity, the set of minimizers is a one-dimensional line segment, of which we choose the midpoint.
However, unlike in the case $p=2$, there is no explicit formula or exact algorithm involving only arithmetic operations and $k$-th roots to compute $\mean$, see \cite{nomedianalg88B}.
Nevertheless, the geometric median can be approximated using Weiszfeld's algorithm, which consists mainly in the fixed point iteration
\[
m^{(k+1)} = \Big( \sumitoN \frac{\lambda_i}{\Vert x_i-m^{(k)}\Vert} \Big)^{-1}\Big( \sumitoN \frac{\lambda_i x_i}{\Vert x_i-m^{(k)}\Vert} \Big),
\]
with a particular choice of the starting point $m^{(0)}$ that guarantees $m^{(k)} \neq x_i$ for all $i=1, \dots, N$ and $k\geq 0$.
This method is a gradient descent method and accelerated methods are also available.
For more details, we refer to the survey \cite{weiszfeld15BS}.

Next, we comment on the relation of $\mmean$ to Wasserstein barycenters.
In the most important case $p=2$, formula \eqref{eq:mmean} simplifies when the transport plans are non-mass-splitting, that is, for every $i=1,\dots, N$, each $\pi^i$ is supported on the graph of some transport maps $T^i\colon \supp(\nu)\to \supp(\mu^i)$ with $T^i_\# \nu = \mu^i$.
In that case, $\mmean$ pushes $\nu$ forward by the average of the transport maps,
\begin{equation}
\mmean
= \Big( \sumitoNlambdai T^i \Big)_\#.
\end{equation}
This is called McCann interpolation for $N=2$.
In the nondiscrete setting, if $\nu$ is absolutely continuous, 
then optimal transport maps $T^i$ exist by Brenier's theorem, 
see e.g.~\cite[Thm.~1.22]{S15otapplied}.
In fact, \cite{ABC16fixedpoint} discusses the following fixed-point iteration for approximate barycenter computation:
\begin{enumerate}
	\item Compute the optimal transport maps $T^i$ from $\nu$ to $\mu^i$, $i=1, \dots, N$
	\item Update $\nu \gets \Big( \sumitoNlambdai T^i \Big)_\# \nu$, repeat.
\end{enumerate}
It is shown that if there is a unique fixed point, 
then this is the optimal barycenter and the iteration converges, which is the case for, e.g., Gaussian measures.
The convergence is numerically observed to be very fast, and in certain special cases, it is reached already in one iteration.
Taking the geometric structure of the Wasserstein space into account, see, e.g., \cite{gradflows08AGS}, the fixed-point procedure above is the the typical algorithm for computing Fr\'echet means on manifolds \cite{frechettemplates05T,frechetpersistence14,frechet19procrustes}. 

This motivates the algorithms presented in this paper, which consist in deliberately performing only the first iteration of the fixed-point procedure above.
More precisely, the approximate barycenters are of the form $\tilde\nu=\mmean(\nu)$ for certain plans $\pi^i$ and initial measures $\nu$.
We found that this yields the best tradeoff between speed and accuracy in practice, since the error improvement of further iterations is typically rather small.

We illustrate this claim by the following numerical example.
We create $N=10$ discrete measures $\mu^i=\sum_{l=1}^n \frac 1 n \delta(x_l^i)$, $i=1, \dots, N$, with $n=50$ points each, which we sample uniformly from the unit disk and center to have mean zero.
We initialize with $\nu^{(0)}\coloneqq \mu^1$ and perform the iteration above until convergence after $5$ iterations, that is, $\nu^{(6)}=\nu^{(5)}$.
Optimal transport maps $T^i$ always exist here, since we have empirical measures with the same number of atoms.
In Figure~\ref{fig:multiple_iterations}, we show the cost $\Psi_2(\nu^{(k)})$ with respect to $k$ and compare to the cost $\Psi_2(\hat\nu)$ of an optimal barycenter $\hat\nu$, that is, a solution of \eqref{eq:bary}.
While the error $\Psi_2(\nu^{(k)})-\Psi_2(\hat\nu)$ is decreased in the first step by $83.2\%$, the improvement in the second iteration is only $37\%$ of the remaining error and decreases even further until convergence to a suboptimal solution.
Moreover, the absolute cost decrease $\Psi_2(\nu^{(2)})-\Psi_2(\nu^{(1)})$ in the second iteration is only $7.5\%$ of the decrease $\Psi_2(\nu^{(1)})-\Psi_2(\nu^{(0)})$ of the first iteration.
This also makes sense intuitively, since it seems reasonable that the largest improvement is gained by pushing every support point from some rather arbitrary initialization to the barycenter of several other reasonably chosen support points of the $\mu^i$.
\renewcommand\curwidth{.7\textwidth}
\begin{figure}[htb]
	\centering
	\includegraphics[width=\curwidth]{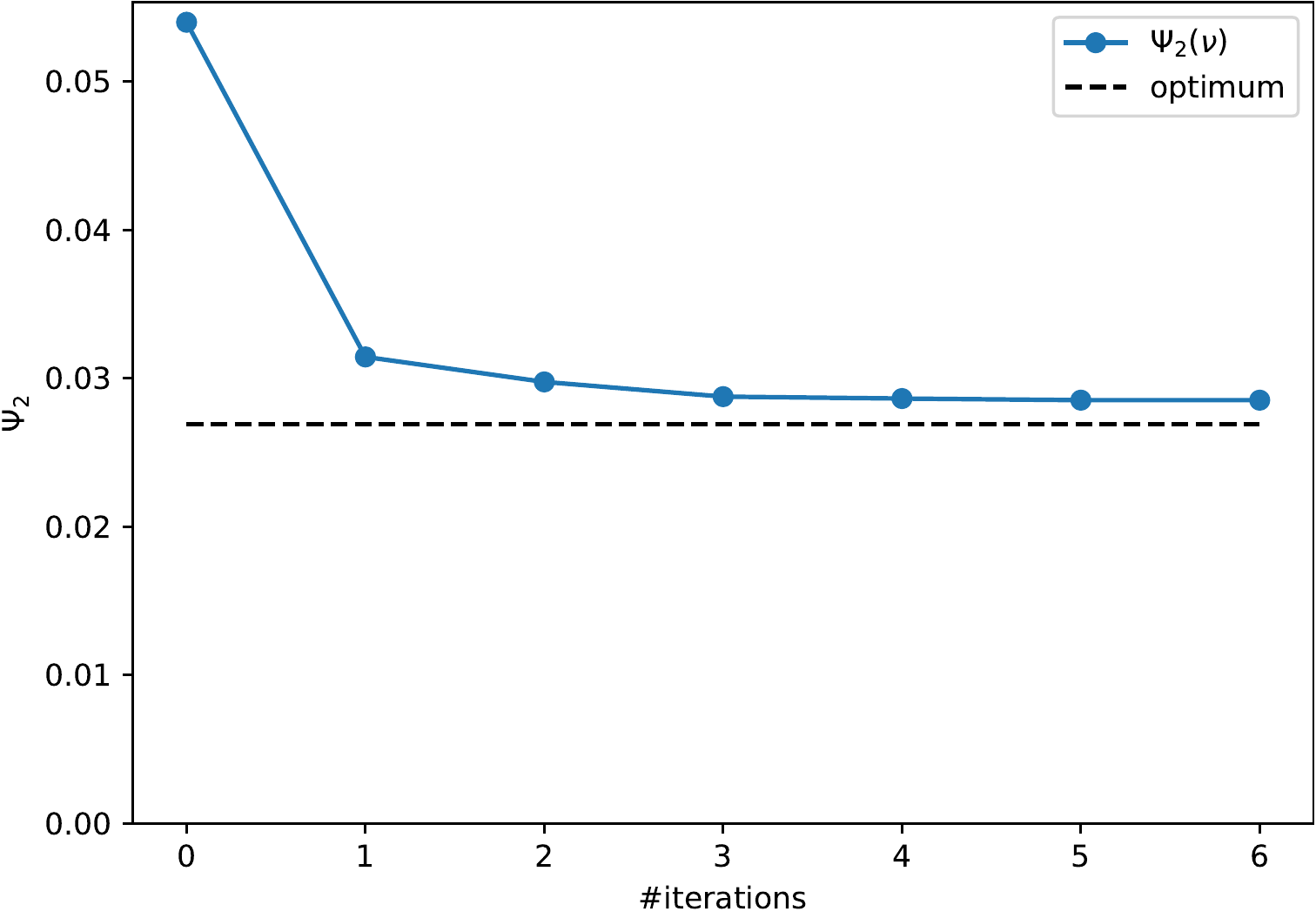}
	\caption{Barycenter cost $\Psi_2(\nu^{(k)})$ over the number of iterations $k$ in blue.
	The black dashed line depicts the optimal cost $\Psi_2(\hat\nu)$.}
	\label{fig:multiple_iterations}
\end{figure}

Furthermore, from a theoretical standpoint, there are simple examples with convergence after one iteration for both presented algorithms below, such that we cannot expect in general to gain any improvements using more than one iteration either.
In particular, as in the numerical example above, there is no way to guarantee convergence to the optimum of this iterative procedure in general, which is the case for any algorithm due to the NP-hardness of the problem \cite{AB21nphard}.

\subsection{Reference Algorithm}\label{sec:ref_alg}
In this section, we choose $\nu=\mu^j$ as initialization.
For simplicity of notation, reorder the measures such that $j=1$.
That is, we compute $N-1$ optimal transport plans
\begin{equation}\label{eq:pii_ref_def}
\pi^i = \sum_{k=1}^{n_1} \sum_{l=1}^{n_i} \pi^i_{k, l} \delta( x^1_k, x_l^i) \in \argmin_{\pi\in\Pi(\mu^1, \mu^i)} \langle c_p, \pi \rangle, \quad i=2, \dots, N
\end{equation}
and consider the approximate barycenter defined by
\begin{equation}\label{eq:ref_bary}
\tilde\nu
= \sum_{k=1}^{n_1} \mu^1_k \delta(\mean (x^1_k)).
\end{equation}
Note that the support of $\tilde\nu$ given by \eqref{eq:ref_bary} is very sparse, since it contains only $n_1$ elements, which is an interesting feature from a computational point of view.

For $p=2$, if the input measures are given in terms of matrices $X^i\in \RR^{n_i\times d}$, where the rows are the support points, and the corresponding mass weights are the vectors $\mu^i\in \RR^{n_i}$ for all $i=1, \dots, N$, then computing the support matrix $Y\in \RR^{n_1\times d}$ of \eqref{eq:ref_bary} can be written as an average of $N$ matrix products as outlined in Algorithm~\ref{alg:ref}.

\begin{algorithm}[htb]
	\begin{algorithmic}
		\State \textbf{Input:} Support points $X^i\in \RR^{n_i\times d}$, masses $\mu^i\in \RR^{n_i}$, $i=1, \dots, N$, weights $0<\lambda\in \RR^N$
		\State $\pi^1\coloneqq \diag(\mu^1)$
		\For{$i=2,\dots, N$}
		\State Compute 
		$\pi^i \in \argmin_{\pi\in\Pi(\mu^1, \mu^i)} \langle c, \pi \rangle \in \RR^{n_1\times n_i}$
		\EndFor
		\State $Y\coloneqq \diag(\mu^1)^{-1} \sumitoNlambdai  \pi^i\cdot X^i$
		\State \textbf{Output:} support $Y\in \RR^{n_1\times d}$, masses $\mu^1\in \RR^{n_1}$
		\caption{Reference algorithm, $p=2$}
		\label{alg:ref}
	\end{algorithmic}
\end{algorithm}
In the case $p=1$, since there is no closed form for $\mean(x_k^1)$, we have to make slight modifications to the algorithm in that case.
\begin{remark}\label{rem:weiszfeld_accuracy}
	Let $f_{x, \lambda}(m) \coloneqq \sumitoNlambdai \Vert x_i - m\Vert$ and $\hat m = \argmin_{m\in \RR^d} f(m)$.
	We show that Weiszfeld's algorithm is guaranteed to approximate $f(\hat m)$ up to a multiplicative factor of $(1+\varepsilon)$ for a certain minimal number of iterations that is explicitly computable.
	In \cite[Thm.~8.2]{weiszfeld15BS} it is shown for the Weiszfeld iterates $m^{(k)}$ that
	\begin{equation}
		f(m^{(k)})-f(\hat m) \leq \frac M k\Vert m^{(0)}-\hat m\Vert^2,
	\end{equation}
	with an explicit formula for $M$, depending only on the $x_i$ and $m^{(0)}$.
	Since $f$ is convex, from $\nabla f = 0$, a simple calculation shows that $\hat m$ must lie in the convex hull of the $x_i$.
	Thus
	\begin{equation}
	\Vert m^{(0)}-\hat m\Vert^2 \leq \max_{i=1, \dots, N} \Vert m^{(0)} - x_i\Vert^2.
	\end{equation}
	Moreover, by \eqref{eq:W1_estimate}, it holds
	\begin{equation}
	\sum_{i<j}^N \lambda_i\lambda_j\Vert x_i-x_j\Vert
	\leq \sumitoNlambdai \Vert x_i - \hat m\Vert
	= f(\hat m),
	\end{equation}
	such that for any given $\varepsilon > 0$, choosing
	\begin{equation}
	k\geq \frac{M\cdot \max_{i=1, \dots, N} \Vert m^{(0)} - x_i\Vert^2}{\varepsilon \sum_{i<j}^N \lambda_i\lambda_j\Vert x_i-x_j\Vert}
	\end{equation}
	guarantees that
	\begin{align}
	f(m^{(k)})
	&= (f(m^{(k)}) - f(\hat m)) + f(\hat m)
	\leq \frac M k \Vert m^{(0)}-\hat m\Vert^2 + f(\hat m) \\
	&\leq \frac{M\varepsilon \sum_{i<j}^N \lambda_i\lambda_j\Vert x_i-x_j\Vert}{M\cdot \max_{i=1, \dots, N} \Vert m^{(0)} - x_i\Vert^2}\Vert m^{(0)}-\hat m\Vert^2 + f(\hat m)
	\leq (1+\varepsilon)f(\hat m).
	\end{align}
\end{remark}
This prepares us to state the reference algorithm for the case $p=1$, see Algorithm~\ref{alg:ref1}.
\begin{algorithm}[htb]
	\begin{algorithmic}
		\State \textbf{Input:} Measures $\mu^i=\sum_{l=1}^{n_i}\mu^i_l\delta(x^i_l)$, $i=1, \dots, N$, weights $\lambda\in \Delta_N$, Weiszfeld accuracy $\varepsilon$
		\State $\pi^1\coloneqq \diag(\mu^1)$
		\For{$i=2,\dots, N$}
		\State Compute 
		$\pi^i \in \argmin_{\pi\in\Pi(\mu^1, \mu^i)} \langle c, \pi \rangle$
		\EndFor
		\For{$k=1, \dots, n_1$}
		\State $m_k\coloneqq \argmin_{m\in \RR^d} \sumitoNlambdai \sum_{l=1}^{n_i} \frac{\pi^i_{k, l}}{\mu^1_k}\Vert m -  x^i_l\Vert$ up to factor $(1+\varepsilon)$, see Remark~\ref{rem:weiszfeld_accuracy}
		\EndFor
		\State \textbf{Output:} approximate barycenter $\tilde\nu\coloneqq \sum_{k=1}^{n_1}\mu^1_k \delta(m_k)$
		\caption{Reference algorithm, $p=1$}
		\label{alg:ref1}
	\end{algorithmic}
\end{algorithm}

\subsection{Pairwise Algorithm}\label{sec:pairwise_alg}

We will see that in order to achieve better results than the reference algorithm, it is beneficial to ``average out'' the asymmetry introduced by choosing $\mu^1$ as the reference measure in \eqref{eq:ref_bary}.
Therefore, we choose
\begin{equation}\label{eq:pairwise_init_choice}
\nu
= \sumitoNlambdai \mu^i
\end{equation}
as initial measure in this section.
However, instead of computing optimal plans from $\nu$ to each $\mu^i$, we solve
\[
\pi^{ij}\in \argmin_{\pi\in\Pi(\mu^i, \mu^j)} \langle c, \pi \rangle
\]
pairwise for every $1\leq i<j\leq N$ and use the transport plans
\begin{equation}\label{eq:pii_pairwise_def}
\pi^i = \sum_{j=1}^N\lambda_j \pi^{ji} \in \Pi(\nu, \mu^i),
\end{equation}
in \eqref{eq:mmean}, so that our approximate barycenter $\tilde\nu$ with \eqref{eq:pairwise_init_choice} and \eqref{eq:pii_pairwise_def} reads as
\begin{equation}\label{eq:pairwise_nu}
\tilde\nu
= \mmean(\nu)
= \sumitoNlambdai \sum_{k=1}^{n_i} \mu^i_k \delta( M_{\lambda, (\pi^{i1}, \dots, \pi^{iN})}(x^i_k)).
\end{equation}
Splitting up the OT computations like this scales better in terms of computational complexity and seems to yield better numerical results in practice.
Clearly, we will have
\[
\#\supp(\tilde\nu) \leq n_1+\dots+n_N,
\]
that is, $\tilde\nu$ meets practically the same sparsity bound as an optimal solution $\hat\nu$, see \eqref{eq:sparsity}.

\begin{remark}
	Note that the inner sum in \eqref{eq:pairwise_nu} is of the form \eqref{eq:ref_bary}.
	If we denote by $\tilde\nu^i$ the barycenter obtained from the reference algorithm, when $\mu^i$ was the reference measure, i.e., permuted to the first position, our approximation \eqref{eq:pairwise_nu} is simply
	\begin{equation}\label{eq:average_of_alg1}
	\tilde\nu = \sumitoNlambdai \tilde\nu^i.
	\end{equation}
	However, since we can choose $\pi^{ji}=(\pi^{ij})^\tT$, we save half of the necessary OT computations compared to executing the reference algorithm $N$ times.
\end{remark}

Algorithm~\ref{alg:pairwise} summarizes this approach for $p=2$ using matrix-vector notation, where $\odot$ denotes element-wise multiplication and $\mathds{1}_d$ denotes a $d$-dimensional vector of ones.
Note that $\eta$ denotes an upper bound of the relative error for the particular given problem, i.e., it holds that $\Psi_2(\tilde\nu)/\Psi_2(\hat\nu)\leq \eta$.
This is proven in Section~\ref{sec:analysis}.

\begin{algorithm}[htb]
	\begin{algorithmic}
		\State \textbf{Input:} Support points $X^i\in \RR^{n_i\times d}$, masses $\mu^i\in \RR^{n_i}$, $i=1, \dots, N$, weights $0<\lambda\in \RR^N$
		\State Set $M=n_1 + \dots + n_N$
		\For{$i=1, \dots, N$}
		\State $\pi^{ii}\coloneqq \diag(\mu^i)$
		\For{$j=i+1,\dots, N$}
		\State Compute 
		$\pi^{ij} \in \argmin_{\pi\in\Pi(\mu^i, \mu^j)} \langle c, \pi \rangle \in \RR^{n_i\times n_j}$
		\State $\pi^{ji}\coloneqq (\pi^{ij})^\tT$
		\EndFor
		\EndFor		
		\bigskip
		\begin{flalign*}
		P&\coloneqq 
		\begin{bmatrix}
		\pi^{11} & \dots & \pi^{1N} \\
		\vdots & \ddots & \vdots \\
		\pi^{N1} & \dots & \pi^{NN}
		\end{bmatrix}
		\in \RR^{M\times M} &&\\
		X&\coloneqq \begin{bmatrix} X^1 \\ \vdots \\ X^N \end{bmatrix} \in \RR^{M\times d} \\
		\Lambda &\coloneqq \begin{bmatrix} \underbrace{\lambda_1, \dots, \lambda_1}_{n_1 \text{ times}}, \dots, \underbrace{\lambda_N, \dots, \lambda_N}_{n_N \text{ times}} \end{bmatrix}^\tT\in \RR^M \\
		\mu &\coloneqq \begin{bmatrix}\mu^1_1 & \dots & \mu^1_{n_1} & \dots & \mu^N_1 & \dots & \mu^N_{n_N}\end{bmatrix}^\tT\in \RR^{M} \\
		Y&\coloneqq \diag(\mu)^{-1} \cdot P \cdot \diag(\Lambda) \cdot X \in\RR^{M\times d} \\
		\nu &\coloneqq \Lambda\odot \mu \in \RR^M\\
		\eta&\coloneqq 2-\nu^\tT\cdot((Y-X)\odot(Y-X))\cdot \mathds{1}_d / \sum_{i<j}^N \lambda_i\lambda_j \langle c_2, \pi^{ij} \rangle \in \RR
		\end{flalign*}
		\State \textbf{Output:} support $Y\in \RR^{M\times d}$, masses $\nu\in \RR^M$, error bound $\eta\in [1, 2]$
		\caption{Pairwise algorithm, $p=2$}
		\label{alg:pairwise}
	\end{algorithmic}
\end{algorithm}

Again, these matrix-vector computations will not work in the case $p=1$.
Instead, Algorithm~\ref{alg:pairwise1} outlines the computation of \eqref{eq:pairwise_nu} using Weiszfeld's algorithm.

\begin{algorithm}[htb]
	\begin{algorithmic}
		\State \textbf{Input:} Measures $\mu^i=\sum_{k=1}^{n_i}\mu^i_k\delta(x^i_k)$, $i=1, \dots, N$, weights $\lambda\in \Delta_N$, Weiszfeld accuracy $\varepsilon$
		\For{$i=1, \dots, N$}
		\State $\pi^{ii}\coloneqq \diag(\mu^i)$
		\For{$j=i+1,\dots, N$}
		\State Compute 
		$\pi^{ij} \in \argmin_{\pi\in\Pi(\mu^i, \mu^j)} \langle c_1, \pi \rangle \in \RR^{n_i\times n_j}$
		\State $\pi^{ji}\coloneqq (\pi^{ij})^\tT$
		\EndFor
		\EndFor	
		\For{$i=1, \dots, N$}
			\For{$k=1, \dots, n_i$}
				\State $m^i_k\coloneqq \argmin_{m\in \RR^d} \sum_{j=1}^N \lambda_j \sum_{l=1}^{n_j} \frac{\pi^{ij}_{k, l}}{\mu^i_k}\Vert m - x^j_l\Vert$ up to factor $(1+\varepsilon)$, see Remark~\ref{rem:weiszfeld_accuracy}
			\EndFor
		\EndFor
		\State $\eta\coloneqq \sumitoNlambdai\sum_{j=1}^N\lambda_j\sum_{k=1}^{n_\nu}\sum_{l=1}^{n_i} \pi^{ij}_{k, l} \Vert m_k^i- x_l^j\Vert / \sum_{i<j}^N \lambda_i\lambda_j \langle c_1, \pi^{ij}\rangle$
		\State \textbf{Output:} approximate barycenter $\tilde\nu\coloneqq \sumitoNlambdai\sum_{k=1}^{n_i}\mu^i_k \delta(m^i_k)$, error bound $\eta\in [1,2]$
		\caption{Pairwise algorithm, $p=1$}
		\label{alg:pairwise1}
	\end{algorithmic}
\end{algorithm}

\section{Analysis}\label{sec:analysis}

In this section, we give worst case bounds for the relative error $\Psi_p(\tilde\nu)/\Psi_p(\hat\nu)$, where $\tilde\nu$ is an approximate barycenter computed by one of the algorithms above, $\hat\nu$ is an optimal barycenter, and $\Psi_p$ is the objective defined in \eqref{eq:bary}.
In the proofs, we will use the following basic identities.

\begin{lemma}\label{lem:basic}
	For any points $x_1, \dots, x_N, y\in \RR^d$, $\lambda\in \Delta_N$ and $m\coloneqq \sumitoNlambdai x_i$, we have the following identities:
	\begin{align}
	\sum_{i=1}^N \lambda_i \Vert x_i - y\Vert^2
	&= \Vert m-y\Vert^2 + \sum_{i=1}^N \lambda_i \Vert x_i - m\Vert^2,
	\label{eq:sunflower} \\
	\sumitoNlambdai \Vert x_i-m\Vert^2
	&= \sum_{i<j}^N \lambda_i\lambda_j \Vert x_i - x_j\Vert^2, \label{eq:star_vs_knn} \\
	\sumitoNlambdai \Vert x_i-y\Vert
	&\geq \sum_{i<j}^N \lambda_i\lambda_j \Vert x_i - x_j\Vert. \label{eq:W1_estimate}
	\end{align}
\end{lemma}
\begin{proof}
	For \eqref{eq:sunflower}, we set $z\coloneqq m-y$ to obtain
	\begin{align}
	\sum_{i=1}^N \lambda_i \Vert x_i - y\Vert^2
	&= \sum_{i=1}^N \lambda_i \Vert x_i -m+z\Vert_2^2
	= \sum_{i=1}^N \lambda_i \left(\Vert z\Vert^2 + \Vert x_i-m\Vert^2 - 2\langle x_i-m, z\rangle \right) \\
	&=  \Vert m-y\Vert^2 + \sum_{i=1}^N \lambda_i \Vert x_i-m\Vert^2.
	\end{align}
	For \eqref{eq:star_vs_knn}, plugging $y=x_j$ into \eqref{eq:sunflower}, we get
	\begin{equation}
		\sumitoNlambdai \Vert x_i-x_j\Vert^2
		= \Vert m-x_j\Vert^2 + \sumitoNlambdai \Vert x_i-m\Vert^2.
	\end{equation}
	Weighting this equality with $\lambda_j$ and summing over $j=1, \dots, N$, we get
	\begin{align}
		\sum_{i,j=1}^N\lambda_i \lambda_j \Vert x_i-x_j\Vert^2
		&= \sum_{j=1}^N \lambda_j\Vert x_j-m\Vert^2 + \sum_{j=1}^N\lambda_j \sumitoNlambdai \Vert x_i-m\Vert^2 \\
		&= 2\sumitoNlambdai \Vert x_i-m\Vert^2.
	\end{align}
	Dividing by $2$ yields \eqref{eq:star_vs_knn}.
	For \eqref{eq:W1_estimate}, note that by the triangle inequality,
	\begin{align}
	\sum_{i<j}^N \lambda_i\lambda_j \Vert x_i - x_j\Vert 
	&= \frac 1 2 \sumitoNlambdai \sum_{j=1}^N\lambda_j \Vert x_i - x_j\Vert \\
	&\leq \frac 1 2 \sumitoNlambdai \sum_{j=1}^N\lambda_j (\Vert x_i-y\Vert + \Vert y- x_j\Vert) \nonumber\\
	&= \frac 1 2 \Big(\sumitoNlambdai \Vert x_i-y\Vert + \sum_{j=1}^N\lambda_j \Vert y- x_j\Vert\Big)
	= \sumitoNlambdai \Vert x_i-y\Vert.\qedhere
	\end{align}
\end{proof}
In order to upper bound $\Psi_p(\tilde\nu)/\Psi_p(\hat\nu)$, we require a lower bound on $\Psi_p(\hat\nu)$.
\begin{proposition}\label{prop:lower_bound_nuhat}
	For any discrete $\nu\in \mathcal P(\RR^d)$ and $p=1, 2$, it holds that
	\begin{equation}\label{eq:lower_bound_nuhat}
	\Psi_p(\nu)
	\geq \sum_{i<j}^N \lambda_i\lambda_j \W_p^p(\mu^i, \mu^j).
	\end{equation}
\end{proposition}
\begin{proof}
	Let $p\in \{1, 2\}$ and $\nu=\sum_{k=1}^{n_\nu} \nu_k \delta(y_k)$ be arbitrary.
	Take $\pi^i \in \argmin_{\pi\in\Pi(\nu, \mu^i)} \langle c_p, \pi \rangle$, then by definition,
	\begin{equation}
	\Psi_p(\nu)
	= \sumitoNlambdai \W_p^p(\nu, \mu^i)
	= \sumitoNlambdai \sum_{k=1}^{n_\nu}\sum_{l_i=1}^{n_i} \pi^i_{k, l_i} \Vert y_k-x^i_{l_i}\Vert^p.
	\end{equation}
	Since it holds for any $i=1, \dots, N$ and $k=1, \dots, n_\nu$ that
	\begin{equation}
	\sum_{l_1, \dots, l_{i-1}, l_{i+1}, \dots, l_N} \frac{\pi^1_{k, l_1}\dots \pi^{i-1}_{k, l_{i-1}}\pi^{i+1}_{k, l_{i+1}}\dots \pi^N_{k, l_N}}{\nu^{N-1}_k}
	= 1,
	\end{equation}
	we get
	\begin{align}
	\Psi_p(\nu)
	&= \sumitoNlambdai\sum_{k=1}^{n_\nu} \sum_{l_1, \dots, l_N} \frac{\pi^1_{k, l_1}\dots \pi^N_{k, l_N}}{\nu^{N-1}_k} \Vert y_k- x^i_{l_i}\Vert^p \\
	&= \sum_{k=1}^{n_\nu} \sum_{l_1, \dots, l_N} \!\!\!\frac{\pi^1_{k, l_1}\dots \pi^N_{k, l_N}}{\nu^{N-1}_k} \sumitoNlambdai\Vert y_k- x^i_{l_i}\Vert^p
	\end{align}
	and by \eqref{eq:star_vs_knn} and \eqref{eq:W1_estimate}, this yields
	\begin{align}
	\Psi_p(\nu)
	&\geq \sum_{k=1}^{n_\nu} \sum_{l_1, \dots, l_N} \!\!\!\!\frac{\pi^1_{k, l_1}\dots \pi^N_{k, l_N}}{\nu^{N-1}_k} \sum_{i<j}^N \lambda_i\lambda_j \Vert x^i_{l_i} - x^j_{l_j}\Vert^p \\
	&= \sum_{i<j}^N \lambda_i\lambda_j \sum_{k=1}^{n_\nu}\sum_{l_i, l_j} \frac{\pi^i_{k, l_i}\pi^j_{k, l_j}}{\nu_k} \Vert x^i_{l_i} - x^j_{l_j}\Vert^p.
	\end{align}
	It is straightforward to check that
	\begin{equation}
	\sum_{k=1}^{n_\nu}\sum_{l_i, l_j} \frac{\pi^i_{k, l_i}\pi^j_{k, l_j}}{\nu_k} \delta(x^i_{l_i}, x^j_{l_j}) \in \Pi(\mu^i, \mu^j),
	\end{equation}
	and so we get
	\[
	\Psi_p(\nu) \geq \sum_{i<j}^N \lambda_i\lambda_j \W_p^p(\mu^i, \mu^j).\qedhere
	\]
\end{proof}

Equipped with \eqref{eq:lower_bound_nuhat}, we can see that already the simple choices $\nu=\mu^j$ and $\nu=\sumitoNlambdai\mu^i$ for the initial measure approximate the optimal barycenter to some extent.

\begin{proposition}\label{prop:init_bounds}
	Let $p\in \{1, 2\}$ and $\hat\nu$ be an optimal barycenter in \eqref{eq:bary}.
	\begin{enumerate}[(i)]
		\item For $\nu\coloneqq \mu^j$, it holds that
		\begin{equation}
		\frac{\Psi_p(\nu)}{\Psi_p(\hat\nu)} \leq \frac{1}{\lambda_j}.
		\end{equation}
		Note that in particular, if $j\in \argmax_{i=1}^N \lambda_i$, then $\Psi_p(\nu)/\Psi_p(\hat\nu)\leq N$.
		\item\label{prop:init_bound_mixture} Let $\nu\coloneqq \sumitoNlambdai \mu^i$, then
		\begin{equation}
			\frac{\Psi_p(\nu)}{\Psi_p(\hat\nu)} \leq 2.
		\end{equation}
		\item If $\nu$ is chosen randomly as one of the $\mu^i$ with probabilities $\lambda_i$, then also
		\begin{equation}
			\frac{\E[\Psi_p(\nu)]}{\Psi_p(\hat\nu)} \leq 2.
		\end{equation}
	\end{enumerate}
\end{proposition}
\begin{proof}
	\begin{enumerate}[(i)]
		\item Let $\nu \coloneqq \mu^j$, then we see that
		\[
		\Psi_p(\nu)
		= \Psi_p(\mu^j)
		= \sumitoNlambdai \W_p^p(\mu^i, \mu^j).
		\]
		By \eqref{eq:lower_bound_nuhat},
		\[
		\Psi_p(\hat\nu)
		\geq \sum_{i<j}^N\lambda_i\lambda_j \W_p^p(\mu^i, \mu^j) \geq \lambda_j \sumitoNlambdai \W_p^p(\mu^i, \mu^j),
		\]
		such that
		\[
		\frac{\Psi_p(\nu)}{\Psi_p(\hat\nu)}
		\leq \frac 1 {\lambda_j}.
		\]
		\item For the choice $\nu\coloneqq\sumitoNlambdai\mu^i$, taking $\pi^{ij}\in\argmin_{\pi \in \Pi (\mu^i, \mu^j)} \langle c_p, \pi\rangle$, we note that
		\[
		\sum_{j=1}^N\lambda_j \pi^{ji} \in \Pi(\nu, \mu^i).
		\]
		Hence,
		\begin{align}
			\Psi_p(\nu)
			&= \sum_{i=1}^N\lambda_i\W_p^p\Big(\sum_{j=1}^N\lambda_j\mu^j, \mu_i\Big)
			\leq \sum_{i=1}^N\lambda_i \langle c_p, \sum_{j=1}^N\lambda_j \pi^{ji}\rangle \\
			&= 2\sum_{i<j}^N\lambda_i\lambda_j\W_p^p(\mu^i, \mu^j),
		\end{align}
		such that
		\[
		\frac{\Psi_p(\nu)}{\Psi_p(\hat\nu)}
		\leq 2.
		\]
		\item This follows similarly as (ii) does by linearity of expectation.\qedhere
	\end{enumerate}
\end{proof}

In general, there is no polynomial-time algorithm that will achieve an error arbitrarily close to $1$ 
with high probability, see \cite{AB21nphard}.
In light of this result, it is interesting to see that it is possible to obtain a relative error bound of $2$ as in \cite{B20LPapprox}, but without performing any computations.
However, note that merely using a mixture of the inputs yields rather useless barycenter approximations in practice; consider, e.g., two distinct Dirac measures.

Although we will see that the bounds above are still more or less sharp for Algorithms~\ref{alg:ref}--\ref{alg:pairwise1}, these algorithms perform a lot better in practice.
Moreover, these bounds are typically drastically improved as soon as a specific problem is given, see Remark~\ref{rem:adapted_bounds} and Section~\ref{sec:numerics}.

Using one of the mentioned trivial choices as initial measures, all algorithms above aim to improve the approximation quality using the mapping $\mmean$.
Next, we show that given any approximate barycenter $\nu$, executing $\mmean$ on $\nu$ never makes the approximation worse, if we choose the OT plans $\pi^i \in \Pi(\nu, \mu^i)$ to be optimal.

\begin{proposition}\label{prop:push_better}
	Given a discrete $\nu=\sum_{k=1}^{n_\nu}\nu_k\delta(y_k)\in \mathcal P(\RR^d)$, $p\geq 1$, let
	\begin{equation}
		\pi^i \coloneqq \sum_{k=1}^{n_\nu}\sum_{l=1}^{n_i} \pi^i_{k, l} \delta(y_k, x^i_l) \in \argmin_{\pi \in \Pi (\nu, \mu^i)} \langle c_p, \pi\rangle
	\end{equation}
	be optimal transport plans.
	Then it holds
	\begin{equation}\label{eq:push_better}
		\Psi_p(\mmean(\nu)) \leq \Psi_p(\nu).
	\end{equation}
\end{proposition}
\begin{proof}
	By definition of $\pi^i$, we have for all $i=1, \dots, N$ that
	\begin{equation}
		\W_p^p(\nu, \mu^i) = \langle c_p, \pi^i\rangle = \sum_{k=1}^{n_\nu}\sum_{l=1}^{n_i}\pi^i_{k, l}\Vert y_k - x_l^i \Vert^p.
	\end{equation}
	Set
	\begin{equation}
		\tilde\pi^i \coloneqq \sum_{k=1}^{n_\nu}\sum_{l=1}^{n_i} \pi^i_{k, l} \delta(m_k, x_l^i) \in \Pi(\mmean(\nu), \mu^i),
	\end{equation}
	where $m_k = \mean(y_k)$.
	Then it holds that
	\begin{align}
		\Psi_p(\mmean(\nu))
		&= \sumitoNlambdai \W_p^p(\mmean(\nu), \mu^i)
		\leq \sumitoNlambdai \langle c_p, \tilde\pi^i\rangle \\
		&= \sumitoNlambdai \sum_{k=1}^{n_\nu}\sum_{l=1}^{n_i}\pi^i_{k, l}\Vert m_k-x_l^i \Vert^p
		= \sum_{k=1}^{n_\nu} \nu_k \sumitoNlambdai\sum_{l=1}^{n_i}\frac{\pi^i_{k, l}}{\nu_k} \Vert m_k-x_l^i\Vert^p \\
		&= \sum_{k=1}^{n_\nu} \nu_k \min_{m\in \RR^d}\sumitoNlambdai\sum_{l=1}^{n_i}\frac{\pi^i_{k, l}}{\nu_k} \Vert m-x_l^i\Vert^p \\
		&\leq \sum_{k=1}^{n_\nu}\nu_k\sumitoNlambdai\sum_{l=1}^{n_i}\frac{\pi^i_{k, l}}{\nu_k} \Vert y_k-x_l^i\Vert^p
		= \sumitoNlambdai \sum_{k=1}^{n_\nu}\sum_{l=1}^{n_i}\pi^i_{k, l} \Vert y_k-x_l^i\Vert^p \\
		&= \sumitoNlambdai \W_p^p(\nu, \mu^i)
		= \Psi_p(\nu).\qedhere
	\end{align}
\end{proof}

Combining the results above, we immediately get the following error bounds for the algorithms introduced in Section~\ref{sec:algorithms}.
\begin{corollary}\label{cor:upper_bounds}
	Let $p\in \{1, 2 \}$ and let $\hat\nu$ be an optimal barycenter.
	\begin{enumerate}[(i)]
		\item If $\tilde\nu$ is obtained by Algorithm~\ref{alg:ref} (case $p=2$) or Algorithm~\ref{alg:ref1} (case $p=1$), then it holds that
		\begin{equation}
			\frac{\Psi_p(\tilde\nu)}{\Psi_p(\hat\nu)} \leq \frac{1}{\lambda_1} \quad \text{or} \quad \frac{\Psi_p(\tilde\nu)}{\Psi_p(\hat\nu)} \leq \frac{1+\varepsilon}{\lambda_1},
		\end{equation}
		respectively.
		Moreover, if instead the reference measure is chosen randomly with probabilities equal to the corresponding $\lambda_i$, then
		\begin{equation}
		\frac{\E[\Psi_p(\tilde\nu)]}{\Psi_p(\hat\nu)} \leq 2
		\quad \text{or}\quad
		\frac{\E[\Psi_p(\tilde\nu)]}{\Psi_p(\hat\nu)} \leq 2(1+\varepsilon).
		\end{equation}
		\item If $\tilde\nu$ is obtained by Algorithm~\ref{alg:pairwise} (case $p=2$) or Algorithm~\ref{alg:pairwise1} (case $p=1$), then it holds that
		\begin{equation}
		\frac{\Psi_p(\tilde\nu)}{\Psi_p(\hat\nu)} \leq 2
		\quad \text{or}\quad
		\frac{\Psi_p(\tilde\nu)}{\Psi_p(\hat\nu)} \leq 2(1+\varepsilon),
		\end{equation}
		respectively.
	\end{enumerate}
\end{corollary}
\begin{proof}
	This follows immediately by combining Propositions~\ref{prop:init_bounds} and \ref{prop:push_better}, and the fact that
	\[
	\sumitoNlambdai \sum_{l=1}^{n_i}\frac{\pi^i_{k, l}}{\nu_k}\Vert m - x_l^i\Vert
	\]
	is only optimized by $m_k$ up to a factor $(1+\varepsilon)$ for every $k=1, \dots, n_\nu$ in the case $p=1$.
\end{proof}

\begin{remark}\label{rem:adapted_bounds}
	Next, we show how to improve on the $2$-approximation bound for a specific given problem.
	We assume that we are given optimal or close to optimal transport plans
	\begin{equation}
		\pi^i
		= \sum_{k=1}^{n_\nu}\sum_{l=1}^{n_i} \pi^i_{k, l} \delta( y_k, x_l^i) \in \Pi (\nu, \mu^i), \quad i=1, \dots, N.
	\end{equation}
	In case of the pairwise algorithm (Algorithms~\ref{alg:pairwise} and \ref{alg:pairwise1}), we use
	\begin{equation}
		\pi^i
		= \sum_{j=1}^N\lambda_j \pi^{ji} \in \Pi(\nu, \mu^i),
		\quad \text{where} \quad
		\pi^{ji}\in \argmin_{\pi\in\Pi(\mu^j, \mu^i)} \langle c_p, \pi \rangle.
	\end{equation}
	Given our approximate barycenter
	\begin{equation}
		\tilde\nu
		= \sum_{k=1}^{n_\nu} \nu_k \delta(m_k), \quad m_k = \mean(y_k),
	\end{equation}
	consider again
	\begin{equation}
		\tilde\pi^i
		\coloneqq  \sum_{k=1}^{n_\nu}\sum_{l=1}^{n_i} \pi^i_{k, l} \delta( m_k, x_l^i) \in \Pi (\tilde\nu, \mu^i).
	\end{equation}
	Then
	\begin{equation}
		\Psi_p(\tilde\nu)
		= \sumitoNlambdai \W_p^p(\tilde\nu, \mu^i)
		\leq \sumitoNlambdai\langle c_p, \tilde\pi^i\rangle
		= \sumitoNlambdai\sum_{k=1}^{n_\nu}\sum_{l=1}^{n_i} \pi^i_{k, l} \Vert m_k- x_l^i\Vert^p.
	\end{equation}
	Together with \eqref{eq:lower_bound_nuhat}, this gives
	\begin{equation}\label{eq:adapted_bound_general}
		\frac{\Psi_p(\tilde\nu)}{\Psi_p(\hat\nu)}
		\leq \frac{\sumitoNlambdai\sum_{k=1}^{n_\nu}\sum_{l=1}^{n_i} \pi^i_{k, l} \Vert m_k- x_l^i\Vert^p}{\sum_{i<j}^N \lambda_i\lambda_j \W_p^p(\mu^i, \mu^j)}.
	\end{equation}
	In the case $p=2$, since
	\begin{equation}
		m_k
		= \mean(y_k)
		= \sumitoNlambdai \sum_{l=1}^{n_i} \frac{\pi^i_{k, l}}{\nu_k} x^i_l
		\quad \text{with} \quad \sumitoNlambdai \sum_{l=1}^{n_i} \frac{\pi^i_{k, l}}{\nu_k} = 1,
	\end{equation}	
	by incorporating \eqref{eq:sunflower}, the denominator in \eqref{eq:adapted_bound_general} simplifies to
	\begin{align*}
		\Psi_2(\tilde\nu)
		&\leq \sum_{k=1}^{n_\nu} \nu_k \sumitoNlambdai\sum_{l=1}^{n_i}\frac{\pi^i_{k, l}}{\nu_k} \Vert m_k-x_l^i\Vert^2 \\
		&= \sum_{k=1}^{n_\nu} \nu_k \Big( \sumitoNlambdai\sum_{l=1}^{n_i}\frac{\pi^i_{k, l}}{\nu_k} \Vert y_k-x_l^i\Vert^2 - \Vert m_k-y_k\Vert^2\Big) \\
		&= \sumitoNlambdai \sum_{k=1}^{n_\nu}\sum_{l=1}^{n_i}\pi^i_{k, l}\Vert m_k-x_l^i\Vert^2 -\sum_{k=1}^{n_\nu} \nu_k\Vert m_k-y_k\Vert^2 \\
		&= \Psi_2(\nu) -\sum_{k=1}^{n_\nu} \nu_k\Vert m_k-y_k\Vert^2,
	\end{align*}
	such that by Proposition~\ref{prop:init_bounds} \eqref{prop:init_bound_mixture}, we get
	\begin{equation}\label{eq:adapted_bound_2}
		\frac{\Psi_2(\tilde\nu)}{\Psi_2(\hat\nu)}
		\leq 2 - \frac{\sum_{k=1}^{n_\nu} \nu_k\Vert  m_k-y_k\Vert^2}{\sum_{i<j}^N\lambda_i\lambda_j \W_2^2(\mu^i, \mu^j)}.
	\end{equation}
	Either way, for both $p=1, 2$, the right-hand sides of \eqref{eq:adapted_bound_general} and \eqref{eq:adapted_bound_2} can be evaluated with almost no computational overhead after the execution of Algorithms~\ref{alg:pairwise} and \ref{alg:pairwise1}, since the optimal transport plans $\pi^{ij}$ between $\mu^i$ and $\mu^j$ have already been computed.
	This usually gives bounds much closer to one than the worst-case guarantees in Corollary~\ref{cor:upper_bounds}.

\end{remark}

Finally, we discuss the sharpness of the bounds in Corollary~\ref{cor:upper_bounds}.
\begin{proposition}\label{prop:sharpness}
	Let $N\geq 2$ and consider the case with $\lambda=(\frac 1 N, \dots, \frac 1 N)\in \Delta_N$.
	There exist measures $\mu^1, \mu^2 = \mu^3 = \dots = \mu^N$, such that if $\hat\nu$ is an optimal barycenter, the following hold true:
	\begin{enumerate}[(i)]
		\item Let $\tilde\nu$ be computed with Algorithm~\ref{alg:ref}, then
		\begin{equation}
			\frac{\Psi_2(\tilde\nu)}{\Psi_2(\hat\nu)}
			= N
			= \frac{1}{\lambda_1}.
		\end{equation}
		If the reference measure is chosen uniformly at random, then
		\begin{equation}
			\frac{\E[\Psi_2(\tilde\nu)]}{\Psi_2(\hat\nu)}
			= 2 - \frac 1 N
			\overset{N\to \infty}{\longrightarrow} 2.
		\end{equation}
		
		\item Let $\tilde\nu$ be computed with Algorithm~\ref{alg:ref1}, then
		\begin{equation}
		\frac{\Psi_1(\tilde\nu)}{\Psi_1(\hat\nu)}
		= N-1
		= \frac 1 {\lambda_1} - 1.
		\end{equation}
		If the reference measure is chosen uniformly at random, then
		\begin{equation}
		\frac{\E[\Psi_1(\tilde\nu)]}{\Psi_1(\hat\nu)}
		= 2\Big(1 - \frac 1 N\Big)
		\overset{N\to \infty}{\longrightarrow} 2.
		\end{equation}
		
		\item Let $\tilde\nu$ be computed with Algorithm~\ref{alg:pairwise}, then
		\begin{equation}
			\frac{\Psi_2(\tilde\nu)}{\Psi_2(\hat\nu)}
			\geq \frac{N-1}{N}\Big( 1+ \frac{N-1}{N} \Big)
			\overset{N\to \infty}{\longrightarrow} 2.
		\end{equation}
		
		\item Let $\tilde\nu$ be computed with Algorithm~\ref{alg:pairwise1}, then
		\begin{equation}
		\frac{\Psi_1(\tilde\nu)}{\Psi_1(\hat\nu)}
		= 2 - \frac 1 N
		\overset{N\to \infty}{\longrightarrow} 2.
		\end{equation}
	\end{enumerate}
\end{proposition}
\begin{proof}
	We consider
	\[
	\mu^1 \coloneqq \delta(0), \qquad
	\mu^2= \ldots = \mu^N \coloneqq \frac 1 2 (\delta(-1)+\delta(1)).
	\]
	\begin{enumerate}[(i)]
		\item For $\pi^i$ defined as in Algorithm~\ref{alg:ref}, it holds 
		\[
		\pi^i = \frac 1 2 (\delta(0, -1) + \delta(0, 1)), \qquad i=2, \dots, N.
		\]
		and thus
		\[
		\tilde\nu
		= \delta\Big(\frac 1 2 (-1+1)\Big)
		= \delta(0)
		= \mu^1.
		\]
		Thus,
		\[
		\Psi_2(\tilde\nu)
		= \sumitoNlambdai \Wtwo^2(\tilde\nu, \mu^i)
		= \frac{N-1}{N}.
		\]
		On the other hand, consider
		\[
		\nu= \frac 1 2 \Big(\delta\Big(-\frac{N-1}{N}\Big) + \delta\Big(\frac{N-1}{N}\Big)\Big),
		\]
		then
		\[
		\Psi_2(\nu)
		= \sumitoNlambdai \Wtwo^2(\nu, \mu^i)
		= \frac 1 N \Big( \Big( \frac{N-1}{N} \Big)^2 + (N-1)\Big( \frac{1}{N} \Big)^2 \Big)
		= \frac{N-1}{N^2},
		\]
		such that
		\[
		\frac{\Psi_2(\tilde\nu)}{\Psi_2(\hat\nu)}
		\geq \frac{\Psi_2(\tilde\nu)}{\Psi_2(\nu)}
		= N
		= \frac 1 {\lambda_1}.
		\]
		\item We only need to compute the following medians:
		\begin{align*}
		&\argmin_{m\in \RR^d} \frac 1 N \Vert 0 - m \Vert + \frac 1 2 \sum_{i=2}^N \frac 1 N (\Vert -1 - m\Vert + \Vert 1 - m\Vert) = 0, \\
		&\argmin_{m\in \RR^d} \frac 1 N \Vert 0 - m \Vert + \sum_{i=2}^N \frac 1 N (\Vert -1 - m\Vert) = -1, \quad \text{and}\\
		&\argmin_{m\in \RR^d} \frac 1 N \Vert 0 - m \Vert + \sum_{i=2}^N \frac 1 N (\Vert 1 - m\Vert) = 1.
		\end{align*}
		Then we see that $\tilde\nu = \mu^1$, such that
		\[
		\Psi_1(\tilde\nu)
		= \sumitoNlambdai \W_1(\mu^1, \mu^i)
		= \frac 1 N \cdot 0 + \Big(1-\frac 1 N\Big) \cdot 1
		= 1-\frac 1 N,
		\]
		and for any $j\in \{ 2, \dots, N \}$,
		\[
		\Psi_1(\mu^j)
		= \frac 1 N\cdot 1 + \Big(1- \frac 1 N\Big)\cdot 0
		= \frac 1 N,
		\]
		which leads to
		\[
		\frac{\Psi_1(\tilde\nu)}{\Psi_1(\hat\nu)}
		\geq 
		\frac{\Psi_1(\tilde\nu)}{\Psi_1(\mu^j)}
		= N-1
		= \frac{1}{\lambda_1}-1.
		\]
		For the randomized case, we get
		\begin{align}
			\mathbb E [\Psi_1(\tilde\nu)]
			&= \frac 1 N \Psi_1(\mu^1) + \Big(1- \frac 1 N\Big) \Psi_1(\mu^j)
			= \frac 1 N\cdot \Big(1- \frac 1 N\Big) + \Big(1- \frac 1 N\Big)\cdot \frac 1 N \\
			&= 2\frac 1 N\Big(1- \frac 1 N\Big),
		\end{align}
		such that
		\[
		\frac{\mathbb E [\Psi_1(\tilde\nu)]}{\Psi_1(\hat\nu)}
		\geq\frac{2\frac 1 N (1- \frac 1 N)}{\frac 1 N}
		= 2\Big(1- \frac 1 N\Big).
		\]
		
		\item We get for $i=2, \dots, N$ that
		\[
		\pi^{ij} = \begin{cases}
		\frac 1 2 (\delta(-1,0)+\delta(1, 0)), & j=1, \\
		\frac 1 2 (\delta(-1,-1)+\delta(1, 1)), & j=2, \dots, N,
		\end{cases}
		\]
		and hence
		\begin{align*}
		\tilde\nu^i
		&= \frac 1 2 \Big( \delta\Big( \frac{N-1}{N}\cdot (-1) + \frac 1 N \cdot 0 \Big) 
		+ \Big( \delta\Big( \frac{N-1}{N}\cdot 1 + \frac 1 N \cdot 0 \Big) \Big) \\
		&= \frac 1 2 \Big( \delta\Big( -\frac{N-1}{N} \Big) + \delta\Big( \frac{N-1}{N} \Big) \Big).
		\end{align*}
		Thus,
		\[
		\tilde\nu
		= \frac 1 N \delta(0) + \frac{N-1}{2N} \Big( \delta\Big( -\frac{N-1}{N} \Big) + \delta\Big( \frac{N-1}{N} \Big) \Big).
		\]
		Hence, it is easy to compute that 
		\[
		\Wtwo^2(\tilde\nu, \mu^i) = \begin{cases}
		\frac{N-1}{N} ( \frac{N-1}{N} )^2 = ( \frac{N-1}{N} )^3, &i=1 \\
		\frac{1}{N} ( \frac{N-1}{N} )^2 + \frac{N-1}{N} ( \frac{1}{N} )^2 = \frac 1 {N^3} N(N-1) = \frac{N-1}{N^2}, &i=2, \dots, N,
		\end{cases}
		\]
		such that
		\[
		\Psi_2(\tilde\nu)
		= \frac 1 N \Big(\Big( \frac{N-1}{N}\Big)^3 + (N-1)\Big(\frac{N-1}{N^2}\Big) \Big)
		= \frac{N-1}{N^2} \Big( \Big(\frac{N-1}{N}\Big)^2 + \frac{N-1}{N} \Big).
		\]
		Finally, considering
		\[
		\nu= \frac 1 2 \Big(\delta\Big(-\frac{N-1}{N}\Big) + \delta\Big(\frac{N-1}{N}\Big)\Big),
		\]
		we get
		\[
		\frac{\Psi_2(\tilde\nu)}{\Psi_2(\hat\nu)}
		\geq 
		\frac{\Psi_2(\tilde\nu)}{\Psi_2(\nu)}
		= \Big(\frac{N-1}{N}\Big)^2 + \frac{N-1}{N}
		= \frac{N-1}{N}\Big( 1 + \frac{N-1}{N}\Big)
		\overset{N\to \infty}{\longrightarrow} 2.
		\]
		
		\item In this case, we get
		\[
		\tilde\nu
		= \frac 1 N \delta(1) + \frac{N-1}{2N}\Big( \delta(-1) + \delta(1) \Big).
		\]
		Compute
		\[
		\Psi_1(\tilde\nu)
		= \sumitoNlambdai \W_1(\tilde\nu, \mu^i)
		= \frac 1 N \cdot \frac{N-1}{N} + \frac{N-1}{N} \cdot \frac 1 N
		= \frac{2(N-1)}{N^2}.
		\]
		On the other hand, for any $j\in \{ 2, \dots, N \}$,
		\[
		\Psi_1(\nu^j)
		= \frac 1 N \W_1(\nu^j, \nu^1)
		= \frac 1 N,
		\]
		such that
		\[
		\frac{\Psi_1(\tilde\nu)}{\Psi_1(\hat\nu)}
		\geq \frac{\Psi_1(\tilde\nu)}{\Psi_1(\nu^j)}
		= 2\frac{\frac{N-1}{N^2}}{\frac 1 N}
		= 2\Big( 1- \frac 1 N \Big)
		\overset{N\to \infty}{\longrightarrow} 2.\qedhere
		\]
	\end{enumerate}
\end{proof}
\begin{remark}
	Intuitively, the example used in the proof of Proposition~\ref{prop:sharpness} is based on the fact that the analyzed algorithms can not split $\mu^1=\delta(0)$ into two Dirac measures with weight $1/2$, in which case the approximations would be optimal.
	We chose the example in the proof for simplicity of exposition.
	However, it is also possible to show the same sharpness results using measures $\mu^1, \dots, \mu^N$ that all have two support points.
	To this end, for $N$ odd and some small $\varepsilon>0$, consider
	\begin{align}
		\mu^1 &\coloneqq \frac 1 2 (\delta(0, -\varepsilon) + \delta(0, \varepsilon)), \\
		\mu^2 = \mu^4 = \dots = \mu^{N-1} &\coloneqq \frac 1 2 (\delta(-1, -\varepsilon) + \delta(1, \varepsilon)) \\
		\mu^3 = \mu^5 = \dots = \mu^N &\coloneqq \frac 1 2 (\delta(-1, \varepsilon) + \delta(1, -\varepsilon)).
	\end{align}
\end{remark}

\section{Numerical Results}\label{sec:numerics}

We present a numerical comparison of different Wasserstein-$2$ barycenter algorithms, the computation of a Wasserstein-$1$ barycenter, and,
as applications, an interpolation between measures and textures, respectively.
To compute the exact two-marginal transport plans of the presented algorithms, we used the \texttt{emd} function of the Python-OT (POT 0.7.0) package \cite{flamary2021pot}, which is a wrapper of the network simplex solver\footnote{\url{https://perso.liris.cnrs.fr/nicolas.bonneel/FastTransport/}} from \cite{BPPH11networksimplex}, which, in turn, is based on an implementation in the LEMON C++ library.\footnote{\url{http://lemon.cs.elte.hu/pub/doc/latest-svn/index.html}}

\subsection{Numerical Comparison}\label{sec:ellipses}

In this section, we compare different Wasserstein-$2$ barycenter algorithms in terms of accuracy and runtime.
We would like to include popular algorithms as iterative Bregman projections into the comparison.
However, many of these algorithms operate in a fixed-support setting, that is, they only optimize over the weights of some a priori chosen support grid.
On the other hand, free-support methods are the ideal candidate for sparse and possibly high-dimensional point cloud data, i.e., if such a grid structure is not present.
An approximation of such data with a coarse grid decreases the accuracy of the solution, but a fine grid increases the runtime of the fixed-support methods.
Hence, the fair choice of a comparison data set is challenging.

We attempt to solve this problem by choosing a grid data set with relatively few nonzero mass weights, that has nevertheless been commonly used as a benchmark example in the literature, also for fixed-support algorithms.
It originates from \cite{CD14fast} and consists of $N=10$ ellipses shown in Figure~\ref{fig:ellipse_data}, given as images of $60\times 60$ pixels.
We take $\lambda\equiv 1/N$.

First, we compute approximate barycenters $\tilde\nu$ using the presented algorithms in the case $p=2$, which we call ``Reference'' and ``Pairwise'' below.\footnote{ \url{https://github.com/jvlindheim/free-support-barycenters}}
Furthermore, we compute the barycenter using publicly available implementations for the methods \cite{JCG20debiased,GWXY19MAAIPM,LSPC2019frankwolfe}, called ``Debiased'', ``IBP'', ``Product'', ``MAAIPM'' and ``Frank--Wolfe'' below,\footnote{ \url{https://github.com/hichamjanati/debiased-ot-barycenters}}
the exact barycenter method from \cite{AB21fixedd} called ``Exact'' below,\footnote{\url{https://github.com/eboix/high_precision_barycenters}} and the method from \cite{LHXCJ20fastIBP} called ``FastIBP'' below.\footnote{ \url{https://github.com/tyDLin/FS-WBP}}
We also tried the BADMM\footnote{ \url{https://github.com/bobye/WBC_Matlab}} method from \cite{YWWL17badmm}, but since it did not converge properly, we do not consider it further.

While the fixed-support methods receive the input measures supported on $\{ 0, \dots, 59/60 \}\times \{ 0, \dots, 59/60 \}$ as gray-valued $60\times 60$ images, the free-support methods get the measures as a list of support positions and corresponding weights.
Clearly, the sparse support of the data is an advantage for the free-support methods.
As a means to facilitate the comparison, we execute the reference and pairwise algorithms also as fixed-support versions.
Instead of computing optimal solutions in Algorithms~\ref{alg:ref} and \ref{alg:pairwise}, we approximate the optimal transport plans $\pi^{ij}$ using the Sinkhorn algorithm on the full grid.
We call these algorithms ``Reference full'' and ``Pairwise full'' below.
Note that, as do the implementations of ``IBP'', ``Debiased'' and ``Product'', we exploit the fact that the Sinkhorn kernel $K=\exp(-c/\varepsilon)$ is separable, such that the corresponding convolution can be performed separately in $x$- and $y$-direction, see, e.g.,~\cite[Rem.~4.17]{PC19book}.
This also reduces memory consumption, since it is not necessary to compute a distance matrix in $\RR^{3600\times 3600}$.
We remark that the runtime of the Sinkhorn algorithms crucially depends on the desired accuracy.
In analogy to ``IBP'', ``Debiased'' and ``Product'' that terminate, once the barycenter measure has a maximum change of $10^{-5}$ in any iteration, we terminate once this tolerance is reached in the first marginal of $\pi^{ij}$.
We check for this criterion only every $10$-th iteration, since it produces computational overhead (contrary to the aforementioned methods).

For all Sinkhorn methods, we used a parameter of $\varepsilon=0.002$ and otherwise chose the default parameters.
For the reference algorithm, we have chosen the reference measure to be the upper left measure shown in Figure~\ref{fig:ellipse_data}.
To compare the runtimes, we executed all codes on the same laptop with Intel i7-8550U CPU and 8GB memory.
The Matlab codes were run in Matlab R2020a.
The runtimes of the Python codes are averages over several runs, as obtained by Python's \texttt{timeit} function.
The results are shown in Figure~\ref{fig:ellipse_barys} and Table~\ref{tab:ellipses}.

\renewcommand\curwidth{\textwidth}
\begin{figure}[htb]
	\centering
	\includegraphics[width=\curwidth]{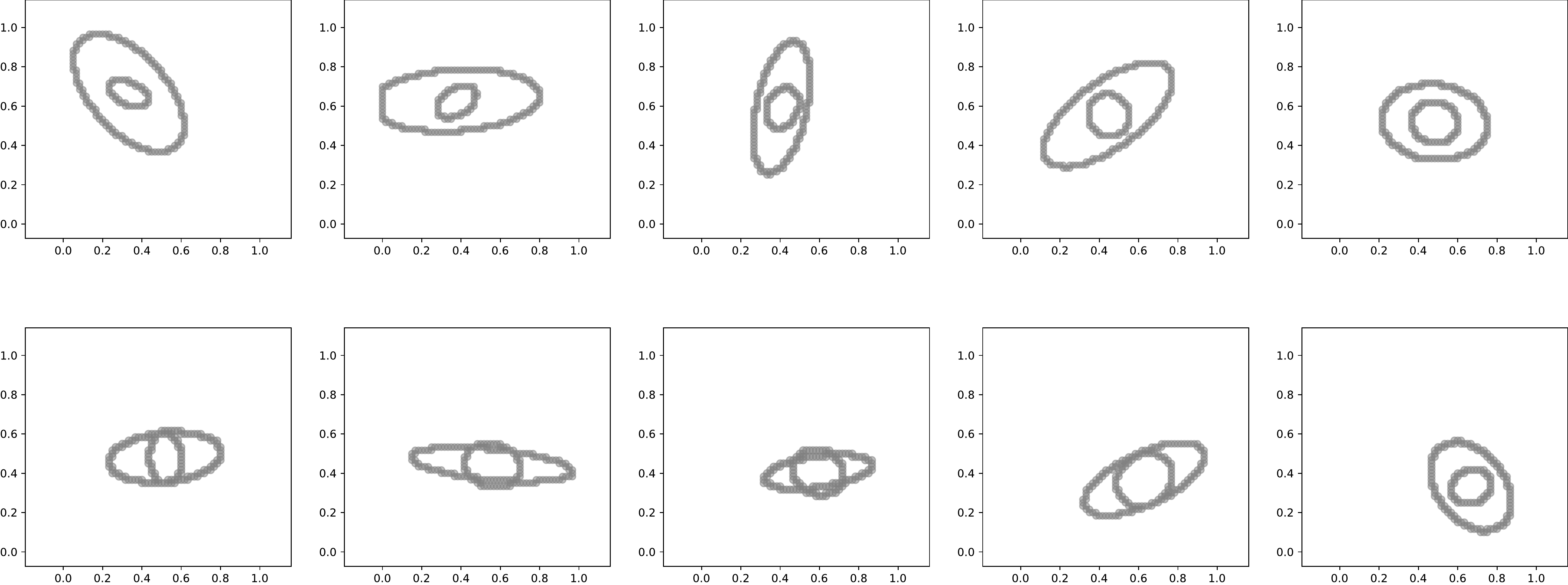}
	\caption{Data set of $10$ nested ellipses. }
	\label{fig:ellipse_data}
\end{figure}
\begin{figure}[htb]
	\centering
	\includegraphics[width=\curwidth]{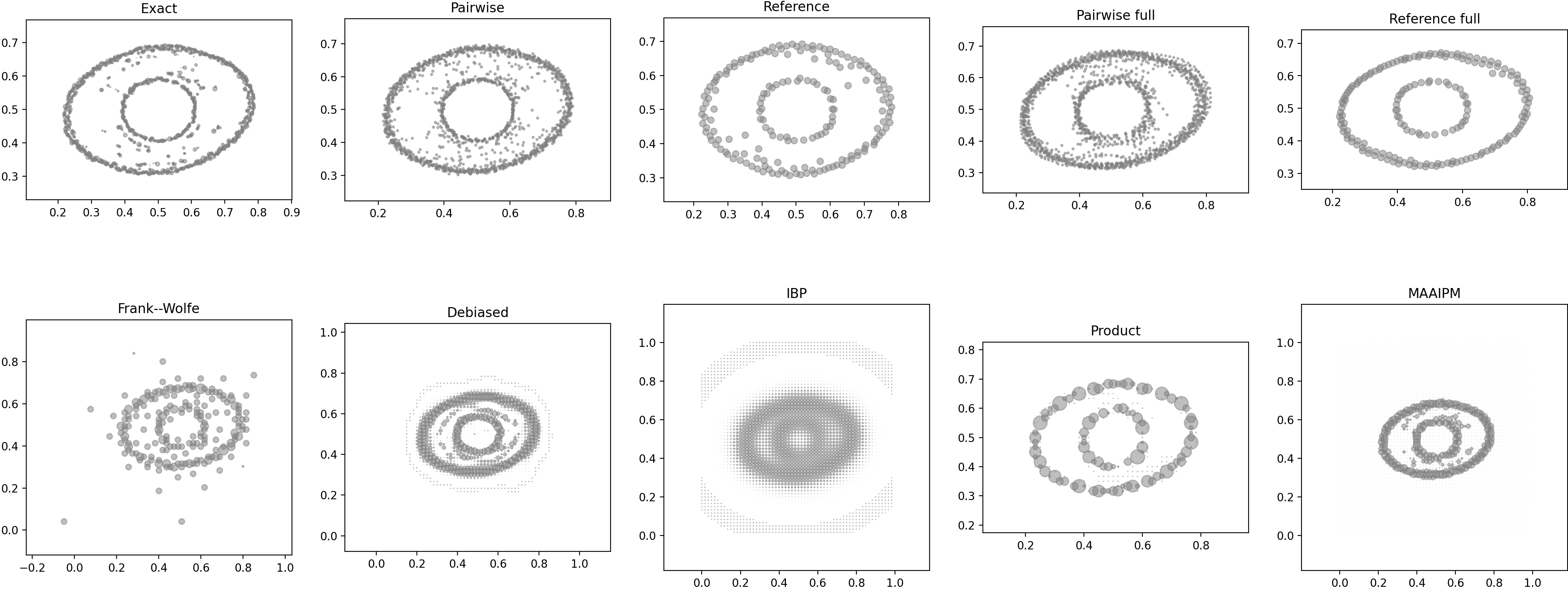}
	\caption{Barycenters for data set in Figure~\ref{fig:ellipse_data} computed by different methods. The weight of a support point is indicated by its area in the plot.}
	\label{fig:ellipse_barys}
\end{figure}

\begin{table}
	\centering
	\begin{tabular}{lrrrrc}
\toprule
{} & $\Psi(\tilde\nu)$ & $\Psi(\tilde\nu)/\Psi(\hat\nu)$ &    runtime & ranking & free support \\
\midrule
Reference      &           0.02683 &                          1.0061 &     \textbf{0.0501} &   \textbf{-1.65} &       \cmark \\
Pairwise       &           0.02669 &                          \textbf{1.0012} &     0.3095 &   -1.41 &       \cmark \\
Pairwise full  &           0.02678 &                          1.0042 &     0.5092 &   -1.14 &       \cmark \\
Debiased       &           0.02675 &                          1.0033 &     1.5061 &   -0.90 &       \xmark \\
Reference full &           0.02716 &                          1.0186 &     0.1128 &   -0.84 &       \cmark \\
IBP            &           0.02723 &                          1.0214 &     0.0914 &   -0.76 &       \xmark \\
Product        &           0.02688 &                          1.0082 &    21.2982 &    0.02 &       \xmark \\
MAAIPM         &           0.02672 &                          1.0020 &   158.5085 &    0.24 &       \xmark \\
Exact          &           0.02666 &                          1.0000 & 18187.6740 &    1.38 &       \cmark \\
FastIBP        &           0.02753 &                          1.0323 &   111.0340 &    1.59 &       \xmark \\
Frank--Wolfe   &           0.02870 &                          1.0763 &    68.2480 &    3.48 &       \cmark \\
\bottomrule
\end{tabular}

	\caption{
		Numerical results for the ellipse barycenter problem.
		The runtime is measured in seconds.
		The ranking is the sum of the standard scores of the logarithm of the relative error and the runtime, respectively. 
		The best values of all approximative algorithms are highlighted in bold.
	}
	\label{tab:ellipses}
\end{table}

While the exact method has a very high runtime, no approximative method achieves a perfect relative error of $\Psi_2(\tilde\nu)/\Psi_2(\hat\nu) = 1$.
However, the error is well below $2$ for all methods, which is a lot better than the worst case bounds shown above.
In fact, using the problem-adapted bounds as outlined in Remark~\ref{rem:adapted_bounds}, without knowledge of $\hat\nu$,
the pairwise algorithm already guarantees a relative error of at most $1.64\%$.
Whereas the pairwise algorithm achieves the lowest error of all approximative algorithms with around $0.12\%$, the reference algorithm achieves the lowest runtime of $0.05$ seconds.
Notably, the FastIBP method is a lot slower than IBP whilst producing a more blurry result, which might indicate an implementation issue.
While the Frank--Wolfe method suffers from outliers, the support of most fixed-support methods is more extended than exact barycenter's support, since Sinkhorn-barycenters have dense support.

We attempt to measure the best compromise between low error and runtime by means of the sum of the standard scores of the logarithmic relative errors and runtimes, respectively, where the standard score or zscore is the value normalized by the population mean and standard deviation.
Table~\ref{tab:ellipses} is sorted according to this ranking score.
The reference and pairwise algorithm are the best with respect to this metric.
As expected, the full-support versions of the reference and pairwise algorithms have worse runtime and also accuracy, which can likely be explained by the errors of the Sinkhorn algorithm.
Nevertheless, they offer a competitive tradeoff between speed and accuracy with respect to the other methods, which shows that the advantage of the framework considered in this paper is not only due to the sparse support of the chosen data set.
Altogether, the results of the proposed algorithms look promising.

\subsection{Wasserstein-$1$ Barycenters}\label{sec:w1_numerics}

Next, we compute approximate Wasserstein-$1$ barycenters of the same data set as in the previous Section~\ref{sec:ellipses} using the Algorithms~\ref{alg:ref1} and \ref{alg:pairwise1}.
The results are depicted in Figure~\ref{fig:ellipse_barys1} in the top row.

Note that the elliptic structure of the barycenter is only retained to some degree, which can probably be explained by the choice of $c_1$ as the cost function.
For example, it is easy to show that the OT plans corresponding to $\W_2^2$ are translation equivariant.
On the other hand, this property fails for any other $p\in [1, 2) \cup (2, \infty]$, as it is easy to derive from the example with $\mu, \nu\in \mathcal P(\RR^2)$ defined by
\begin{equation}
	\mu\coloneqq \frac 1 2 (\delta(0, 0) + \delta(1, 0)), \qquad
	\nu\coloneqq \frac 1 2 (\delta(0, 0) + \delta(0, 1)).
\end{equation}
Thus, we also execute algorithms Algorithms~\ref{alg:ref1} and \ref{alg:pairwise1}, where we swap $c_1$ for the squared Euclidean costs $c_2$ in order to compute the OT plans $\pi^{ij}\in \Pi(\mu^i, \mu^j)$, but continue to compute the barycenter support using Weiszfeld's algorithm.
The results are shown in Figure~\ref{fig:ellipse_barys1} in the bottom row.

\begin{figure}[htb]
	\centering
	\includegraphics[width=.7\curwidth]{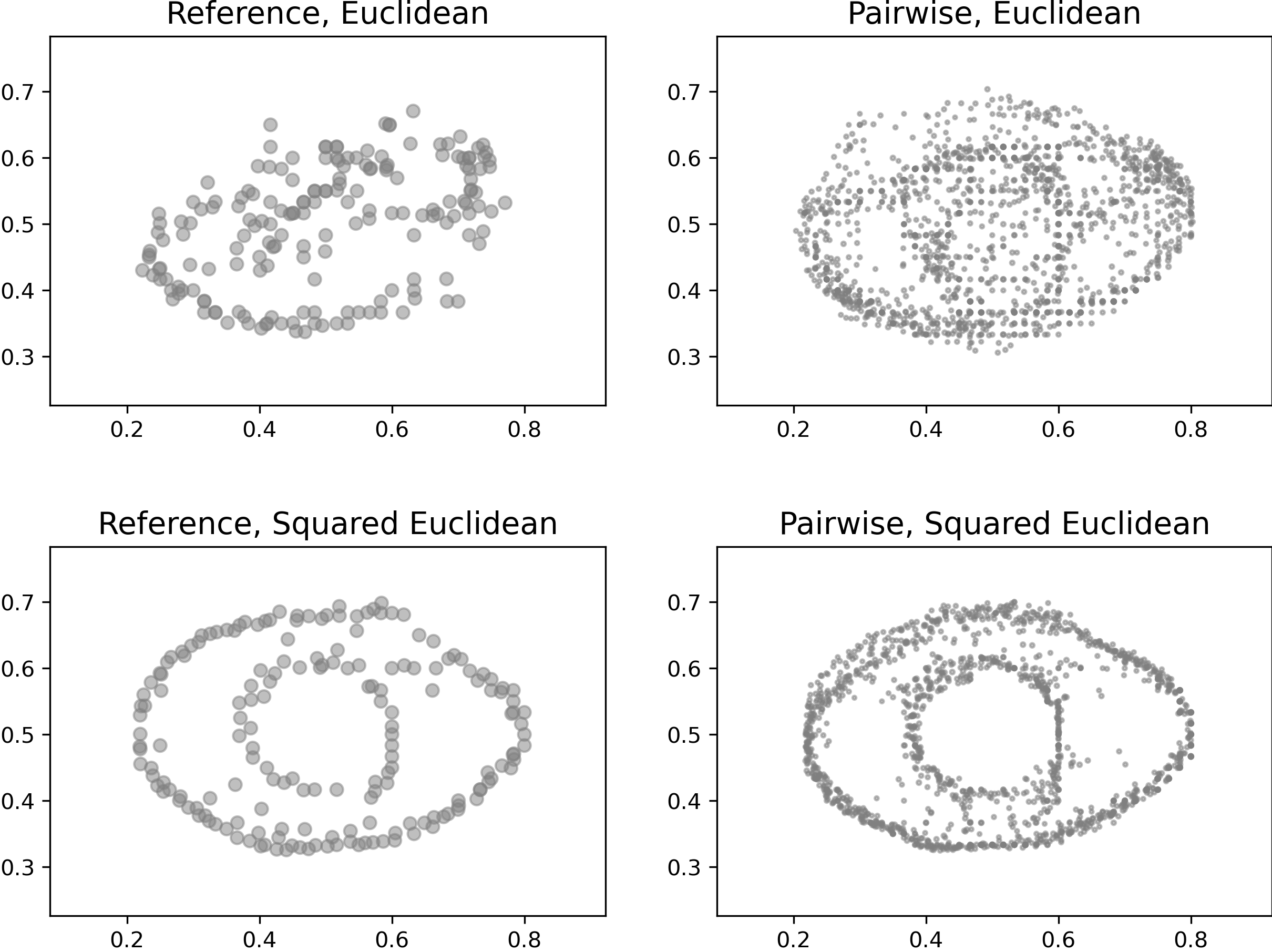}
	\caption{Barycenters computed with Algorithms~\ref{alg:ref1} and \ref{alg:pairwise1} for the data set in Figure~\ref{fig:ellipse_data} and cost functions $c_1(x, y)=\Vert x - y\Vert$ and $c_2(x, y)=\Vert x - y\Vert^2$. The weight of a support point is indicated by its area in the plot.}
	\label{fig:ellipse_barys1}
\end{figure}

Now the elliptic structure is preserved a lot better and the results are very similar to the Wasserstein-$2$ barycenters.
We conclude that the choice of cost function had a larger impact on the results than whether the barycenter support is constructed using the means or geometric medians.
Algorithms~\ref{alg:ref1} and \ref{alg:pairwise1} with $c_2$ thus seem like an interesting alternative to Algorithms~\ref{alg:ref} and \ref{alg:pairwise} in the case where one expects outlier measures, since the median is more robust to outliers than the mean, see, e.g. \cite{breakdown91LR}.

\subsection{Multiple Different Sets of Weights}\label{sec:multiple_weights}

For this numerical application, we compute barycenters between four given measures for multiple sets of weights $\lambda^k=(\lambda_1^k, \dots, \lambda_N^k)$, $\lambda^k\in \Delta_4$, $k=1, \dots, K$, obtaining an interpolation between those measures.
An advantage of the presented algorithms for that application is that the optimal transport plans between the input measures, which are the bottleneck computations, only need to be performed once, whereas the matrix multiplications for interpolations with new weights are fast.
We use the proposed algorithms for a data set of four measures given as images of size $50\times 50$, for sets of weights that bilinearly interpolate between the four unit vectors.
The original measures are shown in the four corners of Figure~\ref{fig:four_images_pairwise}.
For the reference algorithm, we use the upper left measure as the reference measure.
The results are shown in Figures~\ref{fig:four_images_ref} and~\ref{fig:four_images_pairwise}.
\begin{figure}[htb]
	\centering
	\includegraphics[width=\curwidth]{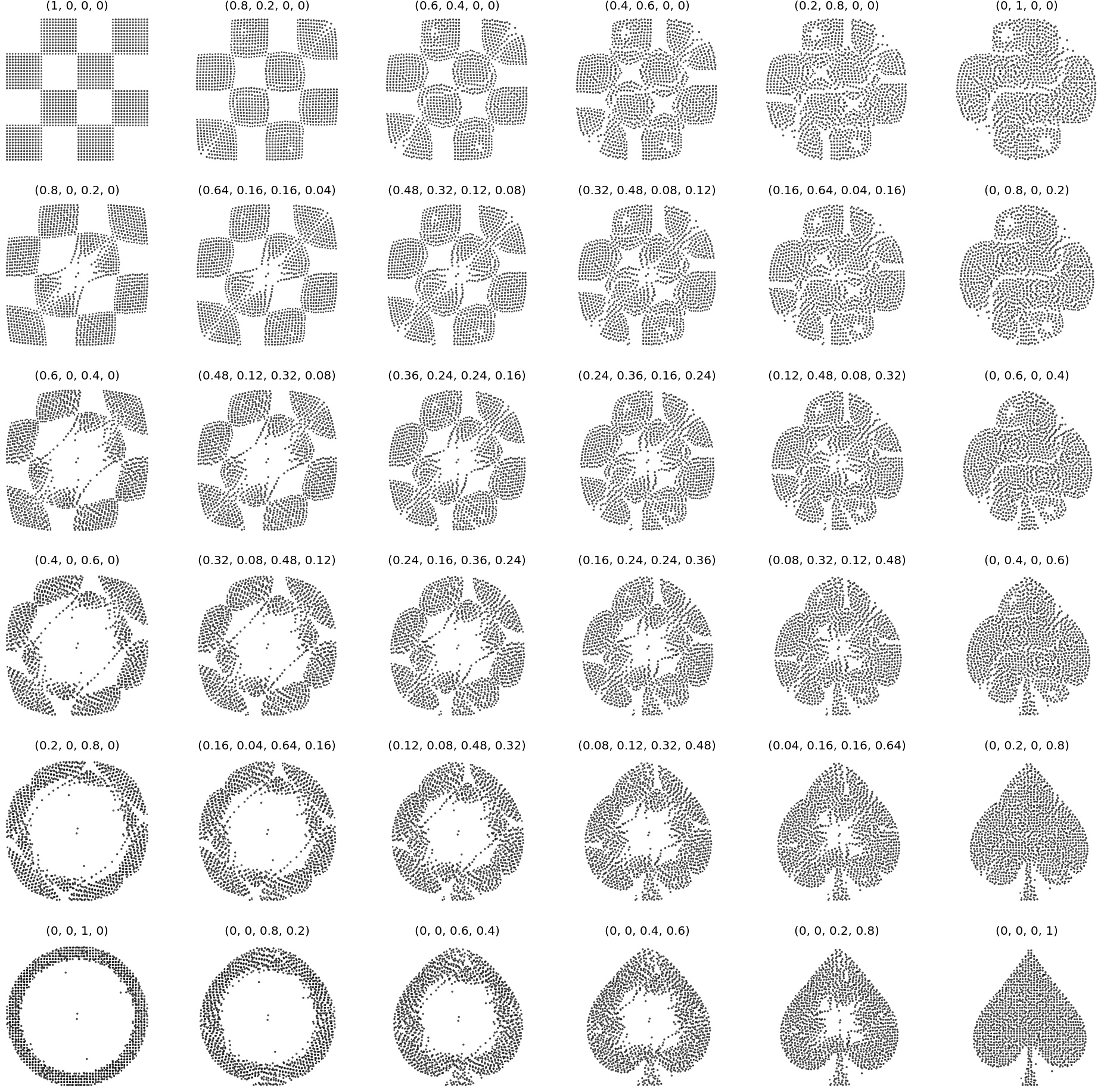}
	\caption{Approximate barycenters for different sets of weights computed by Algorithm~\ref{alg:ref}.}
	\label{fig:four_images_ref}
\end{figure}
\begin{figure}[htb]
	\centering
	\includegraphics[width=\curwidth]{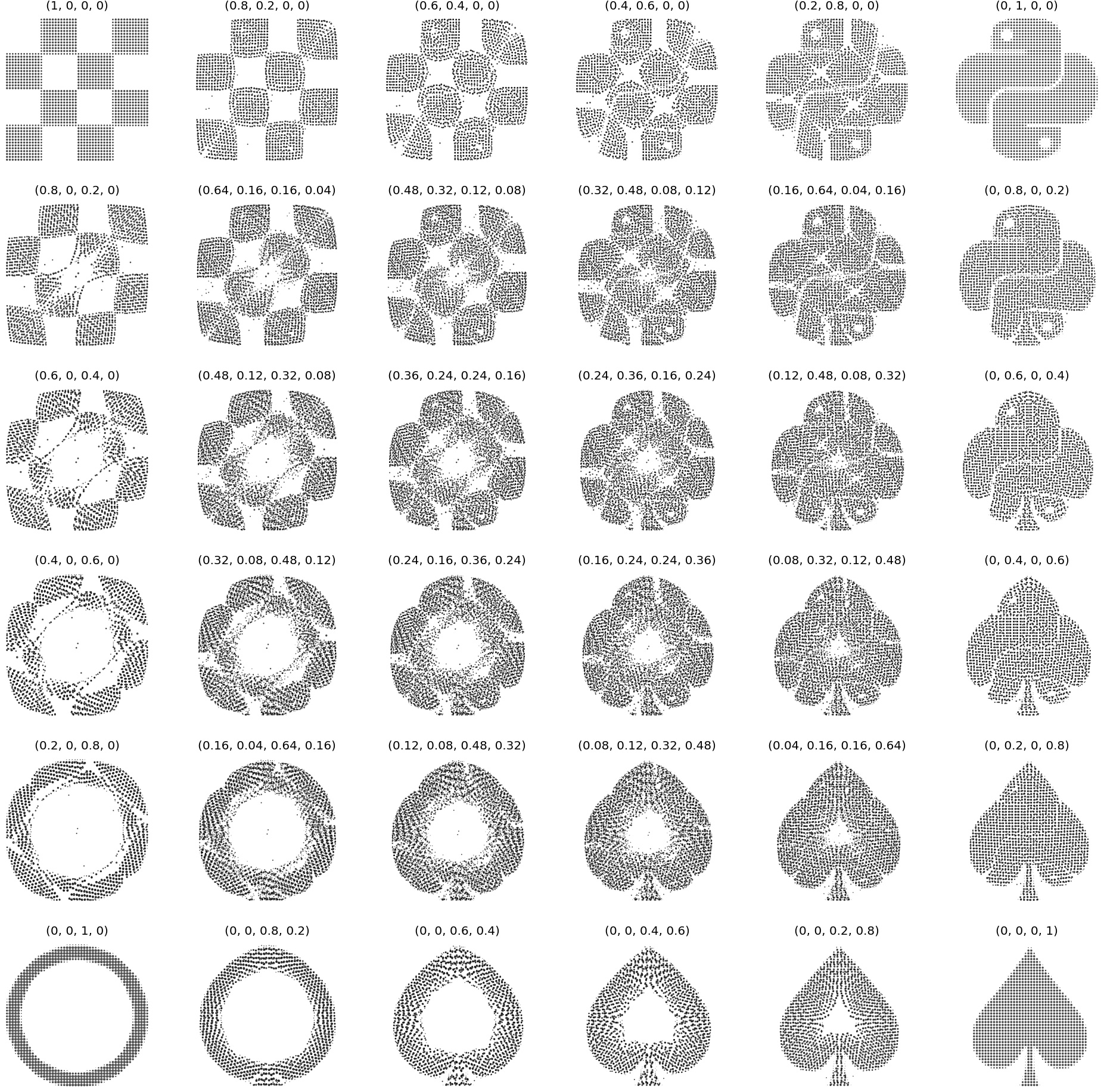}
	\caption{Approximate barycenters for different sets of weights computed by Algorithm~\ref{alg:pairwise}.}
	\label{fig:four_images_pairwise}
\end{figure}

While the running time of the reference algorithm is shorter, its solution has several artifacts, in particular when the weight $\lambda^k_1$ of the reference measure is low.
On the other hand, through effectively averaging the reference algorithm for different choices of the reference measure, the pairwise algorithm is able smooth out some of these artifacts.
We compare the results of both algorithms for $\lambda=(0.04, 0.16, 0.16, 0.64)$ in Figure~\ref{fig:four_images_comparison}.
We also computed the upper error bound $\eta$ of the pairwise algorithm given by \eqref{eq:adapted_bound_2} exemplarily for uniform weights, which is $3.6\%$.

\renewcommand\curwidth{.3\textwidth}
\begin{figure}[htb]
	\centering
	\begin{subfigure}{\curwidth}
		\centering
		\includegraphics[width=\textwidth]{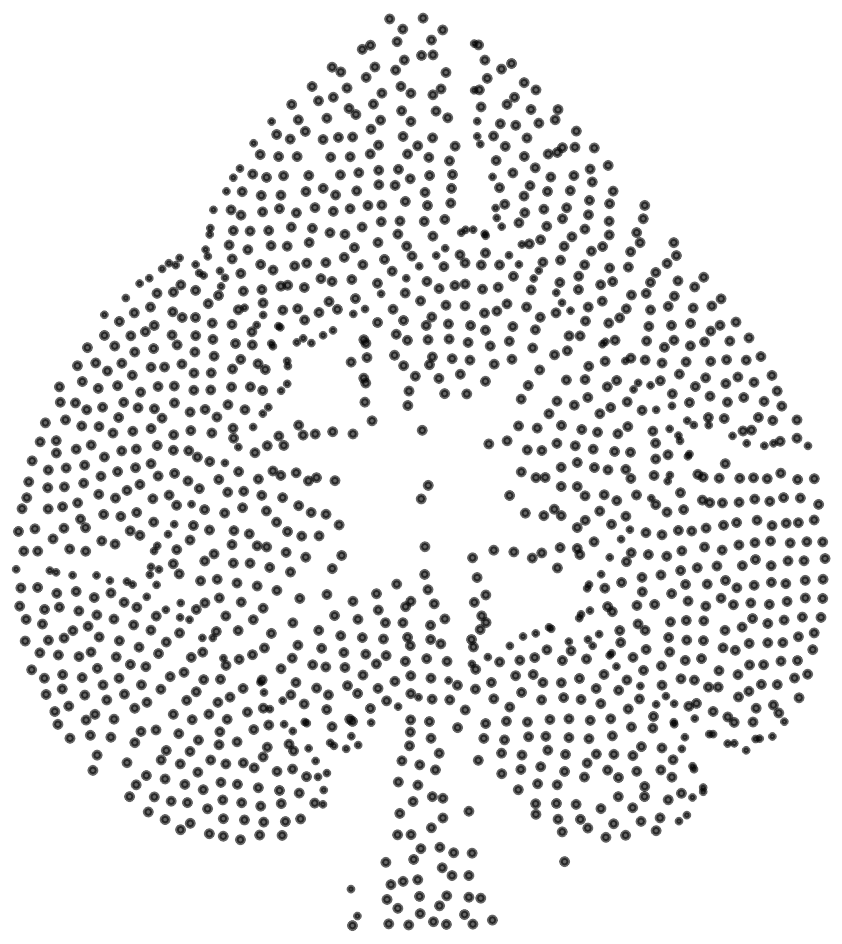}
		\caption{Algorithm~\ref{alg:ref} (reference).}
	\end{subfigure}
	\hspace*{2cm}	
	\begin{subfigure}{\curwidth}
		\centering
		\includegraphics[width=\textwidth]{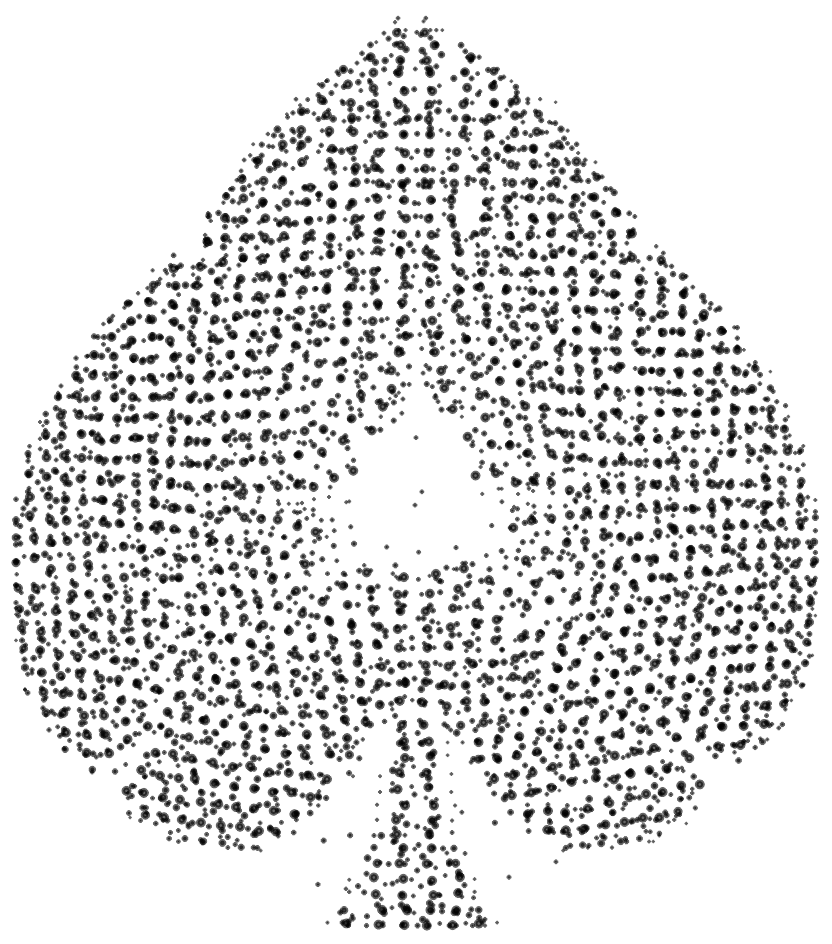}
		\caption{Algorithm~\ref{alg:pairwise} (pairwise).}
	\end{subfigure}
	\caption{Close comparison of two approximate barycenters of the reference and pairwise algorithms for the weights $\lambda=(0.04, 0.16, 0.16, 0.64)$.}
	\label{fig:four_images_comparison}
\end{figure}

\subsection{Texture Interpolation}

For another application, we lift the experiment of Section~\ref{sec:multiple_weights} from interpolation of measures in Euclidean space to interpolation of textures via the synthesis method from \cite{HLPR21gotex}, using their publicly available source code.\footnote{\url{https://github.com/ahoudard/wgenpatex}}
While the authors already interpolated between two different textures in that paper, requiring only the solution of a two-marginal optimal transport problem to obtain a barycenter, we can do this for multiple textures using approximate barycenters for multiple measures.
Briefly, the authors proposed to encode a texture as a collection of smaller patches $F_j$, where each, say, $4\times 4$-patch is encoded as a point $x_j\in \RR^{16}$.
The texture is then modeled as a ``feature measure'' $\frac 1 M \sum_{j=1}^M\delta(x_j)\in \mathcal P(\RR^{16})$,
such that this description is invariant under different positions of its patches within the image.
Finally, this is repeated for image patches at several scales $s$, obtaining a collection of measures $(\mu^s)$, $s=1, \dots, S$.
Synthesizing an image is done by optimizing an optimal transport loss between its feature measure and some reference measure (and then summing over $s$), as obtained, e.g., from a reference image.
Thus, the synthesized image tries to imitate the reference image in terms of its feature measures.
Here, we choose four texture images of size $256\times 256$ from the ``Describable Textures Dataset'' \cite{CMKMV14textures}.
We compute their feature measures $\mu^{1, s}, \dots, \mu^{4, s}$ for each scale.
Next, as in Section~\ref{sec:multiple_weights}, we compute approximate barycenters $\tilde\nu^{k, s}$ for all $k$ and $s$ using the reference algorithm, where $k$ runs over different sets of weights, and perform the image synthesis for each $k$ using the $\tilde\nu^{k, s}$ as feature measures to imitate.
The results are shown in Figure~\ref{fig:textures}.
Using this approach, one obtains a visually pleasing interpolation between the four given textures.

\begin{figure}[htb]
	\centering
	\renewcommand\curwidth{.9\textwidth}
	\includegraphics[width=\curwidth]{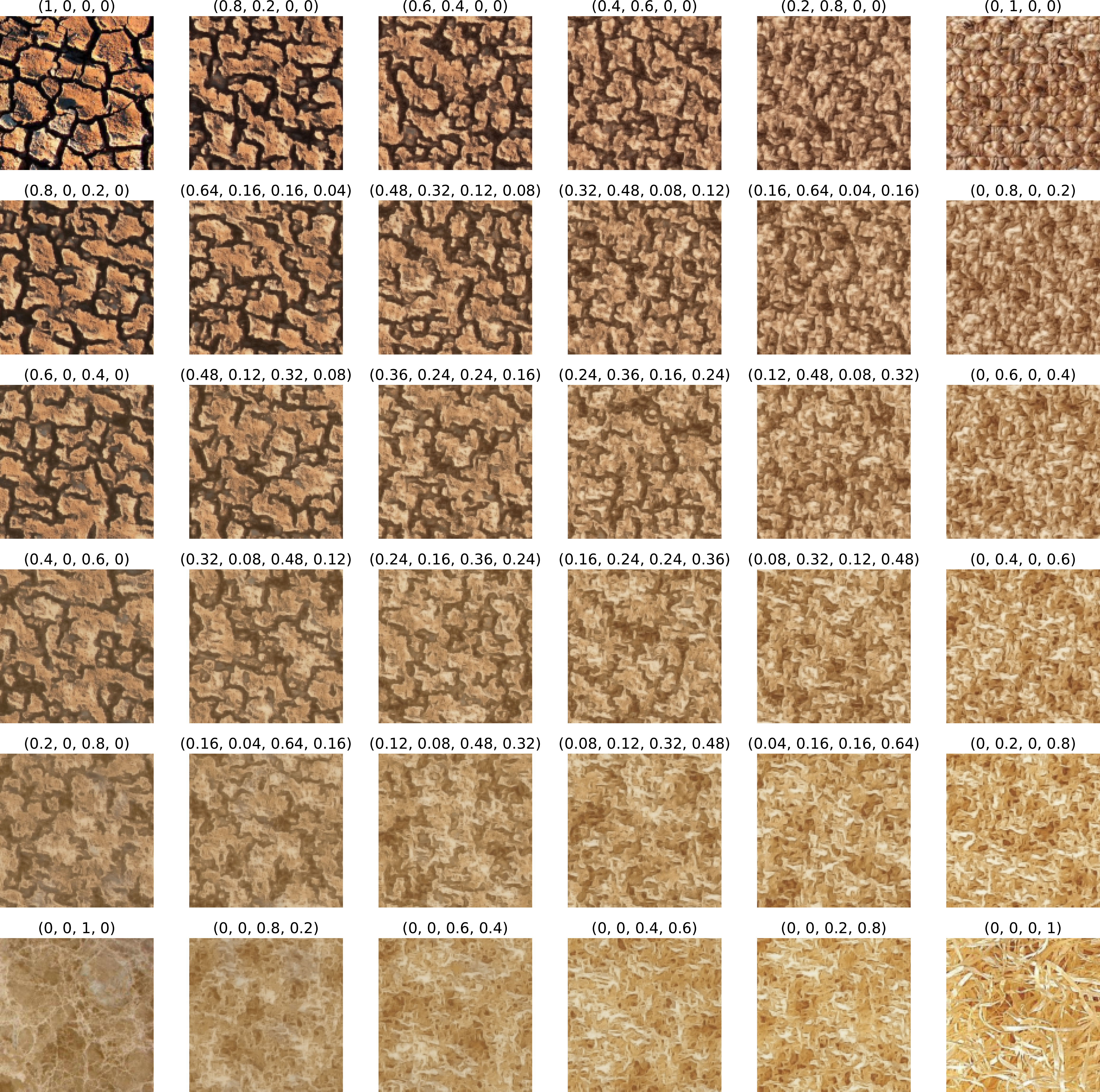}
	\caption{Interpolation of four different textures that are displayed in the four corners. The weight set for the barycenter computations performed for each image is shown above each synthesized image.}
	\label{fig:textures}
\end{figure}

\section{Conclusion}\label{sec:discussion}
In this paper, we derived two straightforward algorithms from a well-known framework for Wasserstein-$p$ barycenters for $p=1, 2$.
We analyzed them theoretically and practically, showing that they are easy to implement, produce sparse solutions and are thus memory-efficient.
We validated their speed and precision using numerical examples.

In the future, it would be interesting to generalize the discussed algorithms and bounds to other $p\geq 1$.
For instance, for $p=\infty$, the barycentric map $\mean$ corresponds to the solution of the so-called smallest-sphere-problem, which can be solved by Welzl's algorithm \cite{welzl91W}.
Finding a lower bound as in Proposition~\ref{prop:lower_bound_nuhat} for general $p\geq 1$ is not straightforward, since the proofs of \eqref{eq:star_vs_knn} and \eqref{eq:W1_estimate} are specific to $p=2$ and $p=1$.

\subsubsection*{Acknowledgements}
Many thanks to Gabriele Steidl and Florian Beier for fruitful discussions.

\subsubsection*{Declarations}
The authors declare no conflict of interest.

\subsubsection*{Data Availability Statement}
The datasets generated during and/or analysed during the current study are available in the GitHub repositories \url{https://github.com/jvlindheim/free-support-barycenters} and \url{https://www.robots.ox.ac.uk/~vgg/data/dtd/}.


\bibliographystyle{abbrv}
\bibliography{literature}

\end{document}